\newif\ifarxiv
\arxivtrue

\ifarxiv
\documentclass[a4paper,          
               12pt,             
               notitlepage]{scrartcl}
\else
\documentclass[final,3p,12pt]{elsarticle}
\journal{Journal of Symbolic Computation}
\fi

\usepackage[utf8]{inputenc}
\usepackage{amsthm}
\usepackage{amsmath}
\usepackage{amsfonts}
\usepackage{amssymb}

\usepackage{hyperref}
\usepackage{url}

\usepackage{tikz}
\usepgflibrary{arrows}

\usepackage{graphicx}
\usepackage{color}

\usepackage{verbatim}

\usepackage{enumitem}
\setlist{noitemsep}

\newtheorem{lemma}{Lemma}
\newtheorem{proposition}[lemma]{Proposition}
\newtheorem{theorem}[lemma]{Theorem}
\newtheorem{corollary}[lemma]{Corollary}

\theoremstyle{definition}
\newtheorem{definition}[lemma]{Definition}
\newtheorem{problem}[lemma]{Problem}
\newtheorem{conjecture}[lemma]{Conjecture}

\theoremstyle{remark}
\newtheorem{remark}[lemma]{Remark}
\newtheorem{example}[lemma]{Example}

\newcommand{\ol}{\overline}

\newcommand{\Kb}{\mathbb{K}}
\newcommand{\Nb}{\mathbb{N}}

\newcommand{\Rb}{\mathbb{R}}
\newcommand{\Zb}{\mathbb{Z}}

\newcommand{\Fbf}{\mathbf{F}}
\newcommand{\Sbf}{\mathbf{S}}

\newcommand{\Acal}{\mathcal{A}}
\newcommand{\Bcal}{\mathcal{B}}

\newcommand{\Gcal}{\mathcal{G}}
\newcommand{\Mcal}{\mathcal{M}}

\newcommand{\Lcal}{\mathcal{L}}

\newcommand{\Fcal}{\mathcal{F}}

\newcommand{\Zker}{\ker_{\Zb}}

\newcommand{\fkc}{\mathfrak{c}}

\newcommand{\zz}{\mathbb{Z}}
\newcommand{\nn}{\mathbb{N}}

\newcommand{\qq}{\mathbb{Q}}
\newcommand{\rr}{\mathbb{R}}

\newcommand{\cala}{\mathcal{A}}
\newcommand{\calb}{\mathcal{B}}

\newcommand{\cf}{\mathcal{F}}

\newcommand{\cm}{\mathcal{M}}

\newcommand{\LL}{\mathcal{L}}

\newcommand{\rank}{\mathrm{rank} \,}

\newcommand{\ind}{\mbox{$\perp \kern-5.5pt \perp$}}

\newcommand{\mardeg}{\mathrm{mardeg}}

\newenvironment{smallpmatrix}{\left(\begin{smallmatrix}}{\end{smallmatrix}\right)}

\DeclareMathOperator{\codim}{codim}
\DeclareMathOperator{\supp}{supp}
\DeclareMathOperator{\ini}{in}

\DeclareMathOperator{\Glues}{Glues}
\DeclareMathOperator{\Lifts}{Lifts}
\DeclareMathOperator{\Quads}{Quads}

\DeclareMathOperator*{\bigtimes}{\textnormal{\Large $\times$}} 

\newcommand{\TFP}[1][\Acal]{\times_{#1}}
\newcommand{\bigTFP}[1][\Acal]{\bigtimes_{#1}}

\newcommand{\Flat}{\Fbf^{\text{\textit{lat}}}}
\newcommand{\Fclat}{\Fcal^{\text{\textit{lat}}}}
\newcommand{\Fbin}{\Fbf^{\text{\textit{in}}}}
\newcommand{\Fcin}{\Fcal^{\text{\textit{in}}}}
\newcommand{\Fblift}{\Fbf^{\text{\textit{lift}}}}
\newcommand{\Fclift}{\Fcal^{\text{\textit{lift}}}}

\begin{document}

\ifarxiv\else
\begin{frontmatter}
\fi

\title{Lifting  Markov Bases and Higher Codimension Toric Fiber Products}

\ifarxiv
\author{Johannes Rauh$^{1,2,}$\footnote{\texttt{jarauh@yorku.ca}}\ $^{,}$\footnote{Present address: York University, 4700 Keele, Toronto, ON M3J 1P3}, Seth Sullivant$^{3,}$\footnote{\texttt{smsulli2@ncsu.edu}}
  \\[5pt]\small
  $^{1}$MPI for Mathematics in the Sciences \\[-10pt]\small
              Inselstraße 22\\[-10pt]\small
              04103 Leipzig, Germany 
  \\[-3pt]\small
  $^{2}$University of Hannover\\[-10pt]\small
              Welfengarten 1\\[-10pt]\small
              30167 Hannover, Germany
              \\\small
           $^{3}$NC Statue University \\[-10pt]\small
           Box 8205 \\[-10pt]\small
           Raleigh, NC 27695 
}

\maketitle
\fi

\begin{abstract}
  We study how to lift Markov bases and Gr\"obner bases along linear maps of lattices.  We give a lifting algorithm that
  allows to compute such bases iteratively provided a certain associated semigroup is normal. 
  Our main application is the toric fiber product of toric ideals, where lifting gives Markov bases of the factor
  ideals that satisfy the compatible projection property.  We illustrate the technique 
  by computing Markov bases of various infinite families of hierarchical models.  The methodology also implies new
  finiteness results for iterated toric fiber products.

  \ifarxiv
  \fi
\end{abstract}

\ifarxiv
\tableofcontents
\else
\end{frontmatter}
\fi

\section{Introduction}
\label{sec:intro}

Let $\Bcal\in\Zb^{h\times n}$ be an integer matrix and let $\Mcal\subseteq\Zb^{n}$.  For any $b\in\Zb^{h}$ let
$\Fbf(\Bcal, b)_{\Mcal}$ be the \emph{fiber graph} with vertex set $\Fbf(\Bcal, b) = \{ v \in \Nb^{n} : \Bcal v = b \}$,
where vertices $u,v\in\Fbf(\Bcal,b)$ are connected by an edge if and only if~$u-v\in\pm\Mcal$.  Then $\Mcal$ is called a
\emph{Markov basis} if all fiber graphs are connected.  Elements of Markov bases are sometimes called \emph{moves}, since they can be used as moves in MCMC simulations to sample from~$\Fbf(\Bcal,b)$~\cite{DiaconisSturmfels1998}.
Alternatively, Markov bases consist of exponent vectors of a binomial generating set of the toric ideal~$I_{\Bcal}$ (see
Theorem~\ref{thm:MBthm}).

The best general algorithm to compute a Markov basis of a matrix is the one
implemented in \verb|4ti2|~\cite{4ti2}.  However, many matrices that appear in applications are too large, and
\verb|4ti2| cannot compute a Markov basis within a reasonable time, using a reasonable amount of memory.
In these situations, one hopes for procedures that take into account the
structure of the Markov basis problem and that can use that structure to
build a Markov basis of a large problem from Markov bases of simpler pieces and ``lifting'' operations.

In this paper we study how to lift a Markov basis along a linear map.
The lifting procedure generalizes similar prior constructions.  For example, the algorithm implemented in \verb|4ti2|
relies on lifting Markov bases along a coordinate projection~\cite{HemmeckeMalkin09:Generating_Sets_Lattice_Ideals}.
The construction used to compute a Markov basis of codimension zero toric fiber products is also an instance
of~lifting~\cite{Sullivant07:TFPs}.  Similar ideas are used in~\cite{Shibuta12:GB_contraction_ideals} to relate an ideal
with its preimage under a monomial ring homomorphism.
We study lifting in a very general context for arbitrary matrices $\Bcal$ and arbitrary linear maps~$\phi$.  The only
assumption that we have to make is that a certain affine semigroup is normal (see Section~\ref{sec:proj-GBs}).  Even if
this condition is violated, in many cases it is possible to adjust our algorithm.  An example is given in
Section~\ref{sec:K4-e}.

Our procedure allows to transform the problem of computing a
Markov basis of~$\Bcal$ into a series of smaller Markov basis computations.  The efficiency of lifting crucially depends
on the choice of the linear map.  If everything goes well, it is possible to compute complicated Markov bases of large
matrices inductively by iterating the lifting procedure.

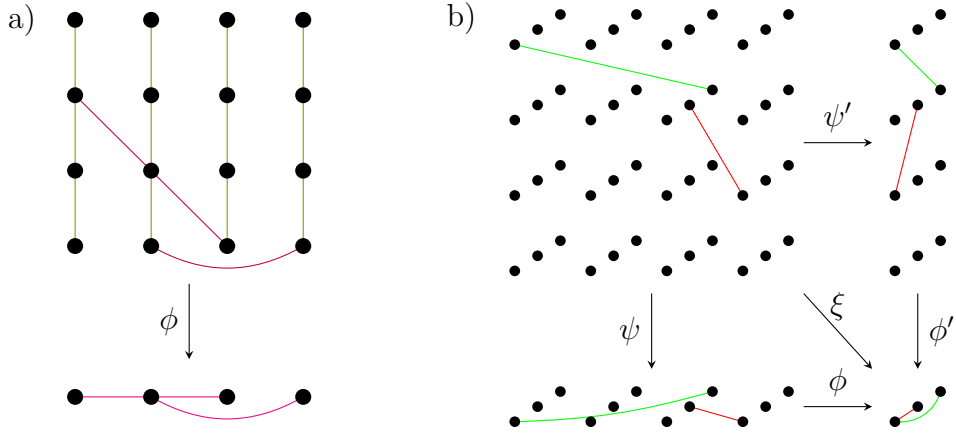
\begin{figure}
  \centering
  \begin{tikzpicture}
    \node at (0.3,4) {a)};
    \foreach \i in {1,...,4} { \foreach \j in {1,...,4} { \path (\i,\j) coordinate (X\i\j); } }
    \begin{scope}[olive]
      \draw (X11) -- (X12) -- (X13) -- (X14); \draw (X21) -- (X22) -- (X23) -- (X24); \draw (X31) -- (X32) -- (X33) --
      (X34); \draw (X41) -- (X42) -- (X43) -- (X44);
    \end{scope}
    \begin{scope}[purple]
      \draw (X13) -- (X22) -- (X31); \draw (X21) edge [bend right=30] (X41);
    \end{scope}
    \foreach \i in {1,...,4} { \foreach \j in {1,...,4} { \fill (X\i\j) circle (3pt); } }
    \path[->,>=stealth] (2.5,0.5) edge node[left] {$\phi$} (2.5,-0.5);
    \foreach \i in {1,...,4} { \path (\i,-1) coordinate (Y\i); }
    \draw[magenta] (2,-1) edge (1,-1) edge [bend right=30] (4,-1) edge (3,-1);
    \foreach \i in {1,...,4} { \fill (Y\i) circle (3pt); }
  \end{tikzpicture}
  \hfil
  \begin{tikzpicture}
    \node at (0.6,4.55) {b)};
    \foreach \i in {1,...,4} { \foreach \j in {1,...,4} { \foreach \k in {1,...,3} { \path (\i+0.3*\k,\j+0.2*\k)
          coordinate (X\i\j\k); } } }
    \draw[red] (X421) -- (X332); \draw[green] (X141) to (X333);
    \foreach \i in {1,...,4} { \foreach \j in {1,...,4} { \foreach \k in {1,...,3} { \fill (X\i\j\k) circle (2pt); } } }
    \foreach \i in {1,...,4} { \foreach \k in {1,...,3} { \path (\i+0.3*\k,-1+0.2*\k) coordinate (Y\i\k); } }
    \draw[red] (Y41) -- (Y32); \draw[green] (Y11) to[bend right=7] (Y33);
    \foreach \i in {1,...,4} { \foreach \k in {1,...,3} { \fill (Y\i\k) circle (2pt); } }
    \foreach \j in {1,...,4} { \foreach \k in {1,...,3} { \path (6+0.3*\k,\j+0.2*\k) coordinate (Yp\j\k); } }
    \draw[red] (Yp21) -- (Yp32); \draw[green] (Yp41) to (Yp33);
    \foreach \j in {1,...,4} { \foreach \k in {1,...,3} { \fill (Yp\j\k) circle (2pt); } }
    \foreach \k in {1,...,3} { \path (6+0.3*\k,-1+0.2*\k) coordinate (Z\k); }
    \draw[red] (Z1) -- (Z2); \draw[green] (Z1) to[bend right=40] (Z3);
    \foreach \k in {1,...,3} { \fill (Z\k) circle (2pt); }
    \path[->,>=stealth] (3.1,0.9) edge node[left] {$\psi$} (3.1,-0.1); \path[->,>=stealth] (6.6,0.9) edge node[right]
    {$\phi'$} (6.6,-0.1); \path[->,>=stealth] (5.1,2.9) edge node[above] {$\psi'$} (6,2.9); \path[->,>=stealth]
    (5.1,-0.6) edge node[above] {$\phi$} (6,-0.6); \path[->,>=stealth] (5.1,0.9) edge node[above] {$\xi$} (6,-0.1);
  \end{tikzpicture}
  \caption{a) Consider a graph $G$ with vertex set $V\subseteq\Zb^n$ and a linear map $\phi:\Zb^n\to\Zb^t$.  If the
    image of $G$ is connected and if each $\phi$-fiber of $G$ is connected, then $G$ itself is connected.
    The \textcolor{magenta}{edges} in the lower graph correspond to a PF Markov basis.  
    The vertical \textcolor{olive}{edges} in the upper graph correspond to a kernel Markov basis.
    The remaining \textcolor{purple}{edges} are lifts of the PF Markov basis.
    b) An illustration of the algorithm applied to the toric fiber product: The goal is to lift along the map~$\xi$.  This can
    be accomplished in two steps, by first lifting along $\phi$ and $\phi'$ and by then gluing the results.}
  \label{fig:Markov-idea}
\end{figure}
The idea behind lifting is sketched in Figure~\ref{fig:Markov-idea}(a): For a linear map $\phi:\Zb^{n}\to\Zb^{t}$ and a
graph $G=(V,E)$ with $V\subseteq\Zb^{n}$ define the image graph $\phi(G)=(V',E')$ by $V'=\phi(V)$ and $(x',y')\in
\phi(G)$ if and only if there is $(x,y)\in E$ with $x'=\phi(x)$ and $y'=\phi(y)$.  If $G$ is a fiber graph of~$\Bcal$
with respect to a Markov basis, then $G$ is connected, and so is~$\phi(G)$.  Our approach is to turn this observation
around as follows: Given a graph homomorphism~$\phi$ induced by a linear map as above, if its image $\phi(G)$ is
connected and if each fiber $G[\phi^{-1}(x)]  := (\phi^{-1}(x), \{ (u,v) \in E :
u,v \in \phi^{-1}(x)\}$ is connected, then $G$ is connected.  Thus, our strategy is as follows:
First, we find a set of moves that connects the $\phi$-fibers (that is the sets of the form $G[\phi^{-1}(x)]$).  Such
a set of moves we call a \emph{kernel Markov basis}, because
$\phi^{-1}(x) = V \cap (u + \ker_{\zz}\phi)$ for any~$u\in\phi^{-1}(x)$.
Second, we find a set of moves that connects the projected fiber graphs~$\phi(G)$.  Such a set of
moves we call a \emph{projected fiber (PF) Markov basis}.  Then we lift the PF Markov basis to obtain suitable moves
in~$\Zb^{n}$.
%
In the last step, the actual lifting step, we need to find ``enough'' preimages of the edges of the image graph.  In
Section~\ref{sec:Lifting-MB-GB} we give a general lifting algorithm of which the central step is again a Markov basis
computation.

The two steps of finding and lifting the PF Markov basis require a generalization of the notion of Markov
basis beyond the one that is typically used in applications.  
This generalized notion is introduced in Section~\ref{sec:markov}.
An important special case is the notion of an inequality Markov basis, which
is a set of moves that connects all generalized fibers of the form
$\{ u\in \LL : Du\ge c\}$  where $\LL$ is a fixed lattice and $D$ is a fixed matrix.
When a certain associated semigroup is normal, the problem of finding
a PF Markov basis can be solved by finding an inequality Markov basis (Section~\ref{sec:proj-GBs}).
  The problem of lifting a Markov
basis element can always be phrased as an inequality Markov
basis problem (Section~\ref{sec:lifting-GBs-algo}).

The main open problem of lifting is how to compute PF Markov bases in the general case,
when the associated affine semigroup 
is not normal.  If 
there are only finitely many holes (or the structure of the holes is sufficiently well understood), our techniques to
compute PF Markov bases can easily be adapted.  However, in general it can be computationally challenging to simply
compute the set of holes of an affine semigroup~\cite{HemmeckeTakemureYoshida09:Computing_holes_in_semigroups}.

Our lifting procedure not only works for Markov bases, but also for the related concept of Gr\"obner bases.  While a
Markov basis connects a set of integer vectors, a Gr\"obner basis allows to find minimal elements with respect to some
suitable order.  For this, the notion of Gr\"obner bases has to be generalized in a similar way as the notion of
Markov bases.  To apply our lifting ideas, we need to assume that the involved orders on the fibers and the
projected fibers are compatible, in a sense that is explained in detail in Section~\ref{sec:Lifting-MB-GB}.

As mentioned before, the complexity of lifting crucially depends on the choice of the map~$\phi$.  In general, we do not
know how to find a good map~$\phi$ for a given Markov basis or Gr\"obner basis problem, or whether such a good map
exists at all.  For hierarchical models, which we study in Section~\ref{sec:models}, it is natural to use
marginalization maps.

\medskip

Our motivation for studying the lifting procedure comes from the study of the toric fiber
product~\cite{Sullivant07:TFPs}.  Let 
$\Acal\in\Zb^{s\times t}$, and denote by $\Nb\Acal$ the affine semigroup generated by the columns~$\Acal$.  The toric
fiber product is a construction that takes two ideals $I,J$ that are homogeneous with respect to an $\Nb\Acal$-grading
and produces another larger ideal~$I\times_{\Acal}J$.  In this paper we focus on the case of toric ideals
$I_{\Bcal},I_{\Bcal'}$ associated with matrices~$\Bcal,\Bcal'$.  In this case, the toric fiber product $I_{\Bcal}\TFP
I_{\Bcal'}$ is again a toric ideal of a matrix~$\Bcal\TFP \Bcal'$, which is the (matrix) toric fiber product of $\Bcal$
and~$\Bcal'$.

The guiding principle in the theory of toric fiber products is that 
the product should inherit many of the nice properties of its factors. 
The complexity of the toric fiber product grows with its \emph{codimension} $\codim\Acal=t-\dim\Nb\Acal$, defined as the
difference between the number of columns of $\Acal$ and the dimension of the semigroup~$\Nb\Acal$.  If the codimension
is zero, then the toric fiber product behaves nicely: For example, Markov bases of $\Bcal$ and
$\Bcal'$ 
can be glued together to Markov bases of
$\Bcal\times\Bcal'$~\cite{Sullivant07:TFPs}, 
and if the two semigroups $\Nb\Bcal$ and $\Nb\Bcal'$ 
are normal, then so
is~$\Nb(\Bcal\TFP\Bcal')$~\cite{EngstromKahleSullivant13:TFP-II} 
(the corresponding statements for ideals also hold for non-toric ideals).  While the codimension one case is more
complicated, still a lot can be said, and if 
$\Bcal,\Bcal'$ are nice enough (specifically, if 
$\Bcal,\Bcal'$ have \emph{slow-varying Markov bases}), 
these can be glued together to produce a Markov basis 
of $\Bcal\TFP\Bcal'$, as shown in~\cite{EngstromKahleSullivant13:TFP-II}.  In~\cite{EngstromKahleSullivant13:TFP-II} it
was also shown that for higher codimensions when the 
Markov bases of $\Bcal$ and $\Bcal'$ satisfy the \emph{compatible projection property}, 
they can be glued together to produce a 
Markov basis of~$\Bcal\TFP\Bcal'$.  Although Markov bases with the compatible projection property always exist,
\cite{EngstromKahleSullivant13:TFP-II} did not give an approach for constructing them.

In the present paper we use our lifting idea to develop a framework to compute compatible Markov bases directly from
scratch as follows; see Figure~\ref{fig:Markov-idea}(b): The $\Nb\Acal$-gradings induce linear projections $\phi,\phi'$
from the fiber graphs of 
$\Bcal,\Bcal'$ to $\Nb^{t}$ (where $t$ is the width of~$\Acal$).  The analogous projection $\xi$ from the fiber graphs
of 
$\Bcal\TFP\Bcal'$ to $\Nb^{t}$ factorizes through these two maps.  We want to lift along~$\xi$.  The first observation
is that a kernel Markov basis $\Mcal_{0}$ of~$\xi$ is given by a corresponding basis of the \emph{associated
  codimension-zero product} (Lemma~\ref{lem:cod-zero-connects-fibers}).  A PF Markov basis can be computed if the
semigroup of the associated codimension-zero product is normal~(Section~\ref{sec:PFI}).  Finally, instead of lifting
along~$\xi$ it is possible to first lift along~$\phi$ and~$\phi'$ and to \emph{glue} the resulting Markov bases of~$I$
and~$J$ (Lemma~\ref{lem:gluing-lifts}).  In the language of~\cite{EngstromKahleSullivant13:TFP-II}, this fact implies
that the lifted Markov bases of $I$ and~$J$ satisfy the compatible projection property.
Our approach generalizes and allows to construct Gr\"obner bases of the toric fiber product.

To sum up, our strategy to compute Markov bases (or Gr\"obner bases) of a toric fiber product $\Bcal\TFP\Bcal'$ 
is as follows:
\begin{enumerate}
\item Find a description of the projected fibers.
\item Find a (generalized) Markov basis $\Gcal$ for this description.
\item Find Markov bases $\Mcal$ and $\Mcal'$ of 
  $\Bcal$ and $\Bcal'$ that lift $\Gcal$.
\item Glue $\Mcal$ and $\Mcal'$ to obtain a Markov basis of the toric fiber product.
\end{enumerate}

Our construction allows for a fairly straightforward
algorithm to produce Mar\-kov bases in many instances where they were
not known before.  We focus in particular in this paper on constructing
Markov bases for hierarchical models where our constructions allow us
to give explicit new instances exhibiting concrete bounds for the
finiteness results that are proven nonconstructively in~\cite{HillarSullivant2012}.
On the other hand, the fact that we cannot work simply with given
Markov or Gr\"obner bases of $\Bcal$ and $\Bcal'$ 
means it is difficult to predict when nice properties of $\Bcal$ and $\Bcal'$ 
are passed on to $\Bcal\TFP\Bcal'$. 
Even so, we provide some examples where a careful analysis allows us to bound degrees of Markov basis elements and
prove normality using the Gr\"obner bases.

\medskip

The paper is organized as follows.  After introducing the
generalized notion of Markov bases and Gr\"obner bases in Section~\ref{sec:markov}, we describe
how to lift Markov and Gr\"obner bases in Section~\ref{sec:Lifting-MB-GB}.
In Section~\ref{sec:TFP} we explain the toric fiber product construction and show how to lift in this case.
Our main motivating examples to study concern Markov bases of hierarchical
models, and we explore these examples in detail in Section \ref{sec:models}.
Section \ref{sec:degree-bounds} explores consequences of the general theory
to producing finiteness results for Markov bases of iterated toric fiber
products, which we apply to deduce finiteness results for Markov bases
of hierarchical models.


\section{Markov bases and Gr\"obner bases of lattice point problems} \label{sec:markov}

We introduce a notion of Markov basis and Gr\"obner basis
for lattice point problems, generalizing the usual notions associated
to integer matrices.
The basic idea is that a Markov basis of a family of sets of integer vectors consists of moves that connects all these
sets.  The usual notion of a Markov basis of a matrix $\Bcal\in\Zb^{h\times n}$ arises by considering the fibers
of~$\Bcal$, where $\Bcal$ is considered as a map $\Nb^{n}\to\Zb^{h}$.  Similary, a Gr\"obner basis of a family of sets
is a set of moves that allows to move towards a minimum on each of these sets, with respect to some order.

Let $\succeq$ be a preorder on $\Zb^{n}$.  Then $\succeq$ is a \emph{total preorder}, if 
for all $u,v\in\Zb^{n}$ either $u\succeq v$ or $v\succeq u$ (or both).  The preorder $\succeq$ is \emph{additive} if it is
total and if $u\succeq v$ implies $u+w\succeq v+w$ for all~$u,v,w\in\Zb^{n}$.  Our main example is the following: Let
$\mathfrak{c} \in \qq^{n}$ and define $\succeq_{\fkc}$ by
\begin{equation*}
  u\succeq_{\fkc} v
  \;:\Longleftrightarrow\;
  \langle \mathfrak{c} , u \rangle \geq \langle \mathfrak{c} , v \rangle.
\end{equation*}
We explicitly allow~$\fkc=0$.  Although the preorder $\succeq_{0}$ is trivial, in the sense that $u\succeq_{0}v$
holds for all $u,v\in\Zb^{n}$, it is useful since it allows a unified treatment of Markov bases and Gr\"obner bases.

More generally, for $\fkc_{1},\dots,\fkc_{r}\in\qq^{n}$, define $\succeq_{\fkc_{1},\dots,\fkc_{r}}$ by
\begin{align*}
  u\succeq_{\fkc_{1},\dots,\fkc_{r}} v 
  &\;:\Longleftrightarrow \;\;
  \langle \mathfrak{c}_{1} , u \rangle > \langle \mathfrak{c}_{1} , v \rangle, \\
  & \quad \text{or }\; \langle \mathfrak{c}_{1} , u \rangle = \langle \mathfrak{c}_{1} , v \rangle \text{ and } \langle \mathfrak{c}_{2} , u \rangle > \langle \mathfrak{c}_{2} , v \rangle,  \\
  & \quad \text{or }\; \langle \mathfrak{c}_{1} , u \rangle = \langle \mathfrak{c}_{1} , v \rangle
       \text{ and } \langle \mathfrak{c}_{2} , u \rangle = \langle \mathfrak{c}_{2} , v \rangle
       \text{ and } \langle \mathfrak{c}_{3} , u \rangle > \langle \mathfrak{c}_{3} , v \rangle, \\
       & \quad \quad\quad\vdots \\
  & \quad \text{or }\; \langle \mathfrak{c}_{1} , (u - v) \rangle = \langle \mathfrak{c}_{2} , (u-v) \rangle = \dots = \langle \fkc_{r-1} ,(u-v) \rangle = 0
  \ifarxiv
  \\ & \qquad\qquad\qquad\qquad\qquad\qquad\qquad\qquad\qquad\qquad\qquad\qquad
  \fi
       \text{ and } \langle \mathfrak{c}_{r} , u \rangle \ge \langle \mathfrak{c}_{r} , v \rangle.
\end{align*}
In fact, any additive preorder is of the form $\succeq_{\fkc_{1},\dots,\fkc_{r}}$~\cite{Robbiano85:Term_orderings}.
Moreover, many additive preorders that appear in practice can be approximated by preorders of
the form $\succeq_{\fkc}$ in a sense to be made precise later (see Remark \ref{rem:weight}).

Fix an additive preorder $\succeq$ on $\Zb^{n}$, let $\Fbf \subseteq \Zb^{n}$ and let $\Mcal \subseteq \zz^{n}$.
Construct a directed graph $\Fbf_{\cm,\succeq}$ with vertex set $\Fbf$ as follows: For $u,v \in \Fbf$ make an edge $u
\to v$ if and only if $v-u \in\pm\Mcal$ and $u \succeq v$.

In a directed graph $G$, declare two vertices $u, v$ equivalent $u \sim v$ if there
is a directed path from $u$ to $v$ and from $v$ to $u$.  The equivalence
classes of $G$ are called the strongly connected components of $G$.
The quotient by the equivalence relation is a directed graph
$G/{\sim}$ that does not contain directed cycles.

%
%
\begin{definition}
  \label{def:MB-GB}
  Let $\Fcal$ be a collection of subsets of $\zz^{n}$.  The set $\cm \subseteq \zz^{n}$ is a \emph{Gr\"obner basis} for
  $\Fcal$ with respect to $\succeq$ if for all $\Fbf \in \Fcal$ the following three conditions are satisfied: 
  \begin{enumerate}
  \item $\Fbf_{\cm, \succeq}$ is weakly connected (i.e.~the underlying undirected graph is connected).  
  \item $\Fbf_{\cm, \succeq}/{\sim}$ contains at most one sink.
  \item Each sink in $\Fbf_{\cm, \succeq}/{\sim}$ is a $\succeq$-minimum.
  \end{enumerate}
  In the special case that ${\succeq}={\succeq_{0}}$ (that is, $u\succeq v$ for all $u,v\in\Zb^{n}$), $\cm$ is called a
  \emph{Markov basis} for~$\Fcal$.  In this case, we also write $\Fbf_{\cm}$ instead of~$\Fbf_{\cm,\succeq_{0}}$.
  Since all edges in $\Fbf_{\Mcal}$ are bidirected, it is convenient to think of $\Fbf_{\cm}$ as an undirected
  graph, and conditions~1, 2 and 3 are equivalent to $\Fbf_{\cm}$ being connected.
\end{definition}
For arbitrary families of subsets~$\Fcal$, the three conditions of Definition~\ref{def:MB-GB} are
independent.
We are mostly interested in the case that all $\Fbf\in \Fcal$ are finite.  In this case, every connected component of
$\Fbf_{\Mcal,\succeq}/\sim$ has a minimum, and thus a sink.  Hence, $\cm$ is a Gr\"obner basis if and only if
$\Fbf_{\cm, \succeq}/{\sim}$ contains precisely one sink for all $\Fbf \in \Fcal$.  In particular, it suffices to only
check the second condition of the definition.  The argument remains true when $\Fbf$ is possibly infinite and $\succeq$
defines a well-ordering on each~$\Fbf$, a case which arises in the context of Gr\"obner bases for lattices with respect
to term orders (see Section~\ref{sec:MB-of-lattices}).

The following lemma follows directly from the definition:
\begin{lemma}
  \label{lem:Markov-is-Grb}
  Let $\Mcal$ be a Markov basis for $\Fcal$.  Then $\Mcal$ is a $\succeq$-Gröbner basis if and only if
  the following two conditions are satisfied:
  \begin{enumerate}
  \item If $u,v\in\Fbf$ are both $\succeq$-minimal, 
    then there is a path from $u$ to $v$ in $\Fbf_{\Mcal,\succeq}$.
  \item If $u\in\Fbf\in\Fcal$ is not $\succ$-minimal in~$\Fbf$, then there
    is a path $u\to u_{1}\to u_{2}\to\dots\to u_{r}$ in $\Fbf_{\Mcal,\succeq}$ with $u_{r}\prec u$.
  \end{enumerate}
\end{lemma}

Gr\"obner bases can be used to find $\succeq$-minimal elements within a set $\Fbf\in\Fcal$.  In particular, a Gr\"obner
basis with respect to $\succeq_{\fkc}$ can be used to solve the integer program
\begin{equation*}
  \text{minimize } \langle \mathfrak{c} , u \rangle  \quad\text { subject to } u \in \Fbf
\end{equation*}
by following the arrows towards the sink equivalence class of $\Fbf$, if it exists.  If there is no
such sink, then the integer program is unbounded, and following the arrows gives an infinite descending sequence.


\subsection{Markov bases and Gr\"obner bases of lattices}
\label{sec:MB-of-lattices}

Let $\calb$ be a $h \times n$ integer matrix and $\Fbf(\Bcal, b)$ denote the
fiber 
$$
\Fbf(\Bcal, b) := \{ v \in \Nb^{n} :  \Bcal v = b  \}.
$$
More generally, if $\LL \subseteq \zz^{n}$ is a lattice (that is, a subgroup of $\zz^{n}$), we can consider fibers
of the form
$$\Flat(\LL,u) :=  \{ v \in \Nb^{n} :  v \in  \LL + u \}.$$
This contains the fibers $\Fbf(\Bcal, b)$ as a subcase, since
$\Fbf(\Bcal, \Bcal u) = \Flat(\ker_{\zz} \Bcal,u)$.

The most commonly studied case of both Markov basis and Gr\"obner basis
arises when
$$\Fcal  = \cf(\Bcal) :=   \{  \Fbf(\Bcal, b)  :  b \in \zz^{h} \}.$$
In this case, Markov bases of $\Fcal$ correspond to binomial generating sets of the associated toric ideal $I_{\calb}$,
and Gr\"obner bases of $\Fcal$ correspond to Gr\"obner bases of~$I_{\calb}$ (see Theorem~\ref{thm:MBthm} below).  Similarly, Markov
bases of
$$\Fclat(\Lcal) :=   \{  \Fbf(\Lcal, u)  :  u \in \zz^{n} \}$$
correspond to generating sets of lattice ideals (see Corollary~\ref{cor:MBthm-for-lattice-ideals} below).

Let $\Kb[x]=\Kb[x_{1},\dots,x_{n}]$ be a polynomial ring.  Any additive preorder $\succeq$ on $\Zb^{n}$
induces a preorder on the monomials in $\Kb[x]$ (denoted by the same symbol) by $x^{u}  \succeq  x^{v}$ if and only if~$u\succeq v$.
If ${\succeq}={\succeq_{\fkc}}$, then this preorder is called the \emph{weight order} induced by~$\mathfrak{c}$. 

For any polynomial $f \in \Kb[x]$, the initial
form of $f$ with respect to $\succeq$, denoted by $\ini_{\succeq}(f)$,
is the sum of all terms $c_{u}x^{u}$ in $f$ such that $u$
is $\succeq$-maximal.  For an ideal $I \subseteq \Kb[x]$, 
\begin{equation*}
  \ini_{\succeq}(I) := \langle {\rm in}_{\succeq}(f) : f \in I \rangle.
\end{equation*}
A set of polynomials $G \subseteq I$ is a Gr\"obner basis for $I$ with respect to
$\succeq$ if and only if
\begin{equation*}
  \langle \ini_{\succeq}(g) : g \in G \rangle  =  \ini_{\succeq}(I).
\end{equation*}
Note that a Gr\"obner basis with respect to $\succeq_{0}$
is nothing but a generating set of~$I$.  

\begin{remark}
  \label{rem:weight}
  Most orders that are used in practice are term orders:
  An additive preorder on $\Zb^{n}$ is called a \emph{term order} if it is a well-ordering; that is $x^{u} \succeq
  x^{v}$ and $x^{v} \succeq x^{u}$ implies $x^{u} = x^{v}$, and every set of monomials has a minimum with respect
  to~$\succeq$.  If $\succeq$ is a term order, then $\ini_{\succeq}(I)$ is a monomial ideal for any~$I$.
 
  For any term order $\succeq$ on $\Kb[x]$ and any ideal $I \subseteq \Kb[x]$, there exists a weight vector $\fkc$ such
  that ${\rm in}_{\succeq}(I) = {\rm in}_{\succeq_{\fkc}}(I) $; see~\cite[Proposition~1.11]{Sturmfels1996:GBCP}.  Hence,
  weight preorders can be used to approximate term orders when working with a fixed ideal.
\end{remark}

For any subset $\Mcal\subseteq\Zb^{n}$ consider the binomial ideal
\begin{equation*}
  I_{\Mcal} := \langle x^{m^{+}} - x^{m^{-}} : m\in\Mcal\rangle,
\end{equation*}
where $m=m^{+}-m^{-}$ is the decomposition of $m$ into its positive and negative part with
$\supp(m^{+})\cap\supp(m^{-})=\emptyset$.  For a lattice $\Lcal\subseteq\Zb^{n}$, the ideal $I_{\Lcal}$ is called a
\emph{lattice ideal}.  If $\Lcal$ is a saturated lattice, that is, if $\Lcal=\Zker\Bcal$ for some integer
matrix~$\Bcal$, then $I_{\Lcal}=:I_{\Bcal}$ is called a \emph{toric ideal}.

\begin{theorem}\cite{DiaconisSturmfels1998,Sturmfels1996:GBCP}
  \label{thm:MBthm}
  A finite subset $\Mcal\subseteq\Zker\Bcal$ is a Markov basis of $\cf(\calb)$ 
  if and only if $I_{\Mcal}=I_{\Bcal}$.
For any additive preorder $\succeq$ on $\Zb^{n}$ for which $0$ is
the smallest element in $\Nb^{n}$, 
a finite subset $\Mcal\subseteq\Zker\Bcal$ is a $\succeq$-Gr\"obner basis of $\cf(\calb)$
if and only if $\{ x^{m^{+}} - x^{m^{-}} : m\in\Mcal \}$ is a $\succeq$-Gr\"obner basis of~$I_{\calb}$.
\end{theorem}

See Theorems 5.3 and 5.5 in \cite{Sturmfels1996:GBCP} for a proof.
Theorem~\ref{thm:MBthm} can easily be generalized to lattice ideals as follows:
\begin{corollary} 
  \label{cor:MBthm-for-lattice-ideals}
  A finite set $\Mcal\subseteq\Lcal$ is a Markov basis of $\Fclat(\Lcal)$ 
  if and only if $I_{\Mcal}=I_{\Lcal}$.
  For any additive preorder $\succeq$ on $\Zb^{n}$ for which $0$ is
the smallest element in $\Nb^{n}$,  a finite subset $\Mcal\subseteq\Lcal$ is a
  $\succeq$-Gr\"obner basis of $\cf(\Lcal)$ if and only if $\{ x^{m^{+}} - x^{m^{-}} :
  m\in\Mcal \}$ is a $\succeq$-Gr\"obner basis of $I_{\Lcal}$.
\end{corollary}
Hilbert's basis theorem implies that finite Markov bases and Gr\"obner bases 
exist for any lattice.
They can be computed using \verb|4ti2|~\cite{4ti2}.


\subsection{Markov bases and Gr\"obner bases of systems of inequalities}

Let $D\in\Zb^{r\times n}$ be an integer matrix and $\LL\subseteq\zz^{n}$ a lattice.  For any $c\in\Zb^{r}$ and $v \in
\zz^{n}$ let
\begin{equation*}
  \Fbin(\LL, v,D,c) := \{  u \in \LL + v  :  Du \geq c  \},
\end{equation*}
and let
\begin{equation*}
  \Fcin(\LL,D) := \big\{ \Fbin(\LL,v,D,c) : c\in\Zb^{r} \mbox{ and } v \in \zz^{n} \big\}.
\end{equation*}
A Markov basis of $\Fcin(\LL,D)$ is called an \emph{inequality Markov basis} of~$\LL$ and $D$ or just an
\emph{$(\LL,D)$-Markov basis}.  If $\succeq$ is an additive preorder, then a $\succeq$-Gr\"obner basis of
$\Fcin(\LL,D)$ is called an \emph{$(\LL,D,\succeq)$-Gr\"obner basis}.
Later, we often choose~$\LL=\Zb^{n}$.  This choice allows us to study linear inequalities over the integers.  In this
case, we suppress $\LL$ from the notation.
  
Inequality Markov bases can be computed by relating them to Markov bases of lattices, which
can be computed in practice using \verb|4ti2|.  We explain this
in the remainder of the section.  The first step is to restrict to
the case that the restriction of the matrix $D$ to the lattice~$\LL$ has rank~$\dim(\LL)$; that is $\LL \cap  \ker_{\zz} D  = 0$.

Let $n'=\rank(D)$, and choose an isomorphism $\kappa:\zz^{n'}\!\cong\zz^{n}/\Zker D$.
Let $\LL'$ be the lattice $\kappa^{-1}((\LL+\Zker D)/\Zker D)$.  Thus, $\LL'\cong\LL/(\LL\cap\Zker D)$.
Let $D'$ be a matrix that represents the induced map 
\begin{equation*}
  \zz^{n'}\to\zz^{r},
  \quad
  u \mapsto D \kappa(u)
\end{equation*}
(note that $\kappa(u)$ is a set that $D$ maps to a single element in~$\zz^{r}$).  Then $\LL'\cap\Zker D'=0$.
There is an isomorphism of lattices
\begin{equation*}
  \iota:\LL \cong \LL' \times (\LL\cap\ker_{\Zb}D).
\end{equation*}
Similarly, for each~$c\in\Zb^{r}$ and $u\in\Zb^{n}$,
\begin{equation*}
  \big\{ u\in\LL + v : D u\ge c\big\} \cong \big\{ u'\in\LL' + v' : D'u'\ge c\big\} \times (\LL\cap\ker_{\Zb}D),
\end{equation*}
where $v'=v+\LL\cap\ker_{\Zb}D$.
In this sense, solving the system of inequalities $D u \ge c$ for $u\in \LL+v$ is equivalent to solving the system $D'
u'\ge c$ for $u'\in \LL'+v'$, and if $\Gcal'\subset\LL'$ is an $(\LL',D')$-Markov basis, then an $(\LL,D)$-Markov basis
is given by
\begin{equation*}
  \big\{ \iota^{-1}(b,0) : b\in\Gcal' \big\}  \cup \{\iota^{-1}(f_{1}),\dots,\iota^{-1}(f_{s})\},
\end{equation*}
where $f_{1},\dots,f_{s}$ generate $\LL\cap\ker_{\Zb}D$.  Conversely, any $(\LL,D)$-Markov basis $\Gcal$ can be
truncated to an $(\LL',D')$-Markov basis $\Gcal'=\iota_{1}(\Gcal)$, where $\iota_{1}:\LL\to\LL'$ denotes the first
component of $\iota$.
Therefore, it suffices to know how to compute inequality Markov bases in the case that the
restriction of $D$ to~$\LL$ has rank~$\dim(\LL)$.  If $L\in\Zb^{n\times t}$ is a matrix such that the columns of $L$ are
a lattice basis of~$\Lcal$, then this is the same as requiring that~$\rank(DL)=t$.
%

\begin{lemma}\label{lem:ineq-Markov}
  Let $D \in \zz^{r \times n}$, let $\Lcal\subseteq\Zb^{n}$ be a lattice, let $L$ be an $(n\times t)$-integer
  matrix such that the columns of $L$ are a lattice basis of~$\Lcal$, and assume that $\tilde D=DL\in\Zb^{r\times t}$
  has rank~$t$.
 \begin{enumerate}
 \item %
   If $\Gcal$ is an $(\LL, D)$-Markov basis, then $D(\Gcal)$ is a Markov basis of the lattice ~$\zz \tilde D$ spanned by
   the columns of~$\tilde D$.
  \item %
    If $\Gcal'$ is a Markov basis of~$\zz \tilde{D}$, then $D^{-1}(\Gcal')\cap\Lcal$ (the inverse image of $\Gcal'$ under the
    linear map corresponding to~$D$) is an $(\LL, D)$-Markov basis.
  \end{enumerate}    
\end{lemma}
\begin{proof}
  It suffices to show that the fibers we want to connect by the $(\LL,D)$-Markov basis
  and by the $\zz \tilde D$-Markov basis are bijective via affine maps with linear part given by~$D$.  By assumption,
  $D$ is invertible on~$\LL$.  Now,
  \begin{align*}
    \Fbin (\LL, v,D,c) &  =  \{  u \in \LL + v  :  Du \geq c  \}
    =     \{  L w + v  : D(Lw + v)  \geq c  \}
    \\ &
    \overset{(1)}{\cong}  \{w  \in  \zz^{t} :   \tilde{D} w  \geq  c - Dv  \}  \\
    & \overset{(2)}{\cong}  \{  (w,w')  \in \zz^{t} \times \nn^{r}  :  \tilde{D} w - w' =  c - Dv   \}  \\
    & \overset{(3)}{\cong}  \{  w' \in \nn^{r}  :   w'  \in  \zz\tilde{D} + Dv -c  \}
    =   \Flat(\zz \tilde{D}, Dv-c ).
  \end{align*}
  The bijection (1) arises from multiplication by a left-inverse of~$L$.  The bijections (2) and (3) arise from the
  linear projections from $(w,w')$ to either the first or second coordinate.
  In total, the resulting map from $\Fbin(\LL, v,D,c)$ to $\Flat(\zz \tilde{D}, Dv-c )$ is given by $u\mapsto \tilde
  D\ol L(u-v) - c + Dv$, where $\ol L$ is a left-inverse of~$L$.  By assumption, $u-v\in\Lcal$, and hence $\tilde D\ol
  L(u-v)=D(u-v)$.  Therefore, for any values of $c$ and $v$, the bijection between $\Fbin(\LL, v,D,c)$ and $\Flat(\zz
  \tilde{D}, Dv-c)$ is an affine map with linear part given by~$D$.
\end{proof}

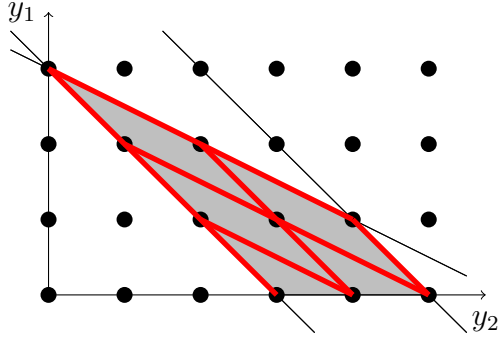
\begin{figure}
  \centering
  \begin{tikzpicture}
    \foreach \i in {0,...,7} { \foreach \j in {0,...,5} { \path (\i,\j) coordinate (X\i\j); } }
    \begin{scope}[fill=lightgray]
      \filldraw (6,1) -- (5,2) -- (1,4) -- (4,1) -- cycle;
    \end{scope}
    \foreach \i in {1,...,6} { \foreach \j in {1,...,4} { \fill (X\i\j) circle (3pt); } }
    \begin{scope}[->]
      \draw (1,1) -- (6.75,1);
      \draw (1,1) -- (1,4.75);
    \end{scope}
    \node at (6.75,0.65) {$y_{2}$};
    \node at (0.65,4.75) {$y_{1}$};
    \filldraw (6.5,0.5) -- (2.5,4.5);
    \filldraw (6.5,1.25) -- (0.5,4.25);
    \filldraw (4.5,0.5) -- (0.5,4.5);
    \begin{scope}[line width=2pt,red] 
      \draw (X61) -- (X52); \draw (X51) -- (X42) -- (X33); \draw (X41) -- (X32) -- (X23) -- (X14);
      \draw (X52) -- (X33) -- (X14); \draw (X61) -- (X42) -- (X23); \draw (X51) -- (X32);
    \end{scope}
  \end{tikzpicture}
  
  \caption{The set of solutions to~\eqref{eq:simple-inequalities} for $c=(0,5,3,6)$.  The
    red edges correspond to the moves in the Markov basis~\eqref{eq:simple-iMB}.}
\label{fig:simple-iMB}
\end{figure}

\begin{example}  \label{ex:simple-iMB}
  Suppose we want to compute an inequality Markov basis of 
    \begin{equation*}
    D =
    \scalebox{0.9}{$
    \begin{pmatrix}
      1 &  0 \\
      -1 & -1 \\
      1 &  1 \\
      -2 & -1
    \end{pmatrix}$},
  \end{equation*}
  that is, we want to obtain a set of moves that connects all integer points $(y_{1}, y_{2})$
  that satisfy
  \begin{equation}
    \label{eq:simple-inequalities}
    y_{1} \ge c_{1},
    \qquad
    y_{1} + y_{2} \le c_{2},
    \qquad
    y_{1} + y_{2} \ge c_{3},
    \qquad
    2 y_{1} + y_{2} \le c_{4}
  \end{equation}
  for any $c_{1}, c_{2}, c_{3}, $ and $c_{4}$.
  The two columns of $D$ span a two-dimensional lattice~$\Zb D\subset\Zb^{3}$.  By \verb|4ti2|, a Markov basis of $\Zb D$ is
  given by
  \begin{equation*}
    \Gcal' = \Big\{ (1, 0, 0, -1), \quad (1, 1, -1, 0) \Big\}.
  \end{equation*}
  The inverse image of $\Gcal'$ under $D$ is
  \begin{equation}
    \label{eq:simple-iMB}
    D^{-1}\Gcal' = \Big\{ (1, -1), \quad (1, -2) \Big\}.
  \end{equation}
  Lemma~\ref{lem:ineq-Markov} (with $\Lcal=\Zb^{2}$ and $\tilde D=D$) implies that this is an inequality Markov basis.
  The situation is visualized in Figure~\ref{fig:simple-iMB}.  \qed
\end{example}

\begin{example}
  \label{ex:simple-ilMB}
  Suppose we want to compute an inequality Markov bases for the following
  system of equations and inequalities:
  \begin{equation}
    \label{eq:simple-ineqs-l}
    \begin{gathered}
      y_{1} + y_{2} + y_{3} = 0,
      \\
      y_{1} \ge c_{1},
      \qquad
      y_{3} \ge c_{2},
      \qquad
      y_{1} + y_{2} \ge c_{3},
      \qquad
      y_{1} - y_{3} \le c_{4}.
    \end{gathered}
  \end{equation}
  One possibility to study this system is to replace the first equation by the two inequalities
  \begin{equation*}
    y_{1} + y_{2} + y_{3} \ge 0
    \quad\text{ and }\quad
    y_{1} + y_{2} + y_{3} \le 0.
  \end{equation*}
  This leads to a matrix $D'$ of size~$6\times 3$ and a Markov basis of cardinality~3.  Alternatively, one can observe that the first equation defines a
  lattice~$\Lcal$, which is generated by the columns of
  \begin{equation*}
    L =
    \scalebox{0.9}{$
    \begin{pmatrix}
      1 & 0 \\
      0 & 1 \\
      -1 & -1
    \end{pmatrix}$}.
  \end{equation*}
  This choice of $L$ corresponds to eliminating $y_{3}$ from~\eqref{eq:simple-ineqs-l} 
  and leads to the same system of inequalities as in Example~\ref{ex:simple-iMB}.  The matrix $LD'$ is equal to the
  matrix $D$ augmented by two rows of zeros.  By Lemma~\ref{lem:ineq-Markov}, the set
  \begin{equation}
    \label{eq:simple-ilMB}
    \Gcal = \Big\{ (1, -1, 0), \quad (1, -2, 1) \Big\}
  \end{equation}
  is a Markov basis of~\eqref{eq:simple-ineqs-l}.  \qed
\end{example}

To construct Gr\"obner bases of toric fiber products, we also need to find \emph{inequality Gr\"obner bases}
for the family~$\Fcin(\LL,D)$.  Such Gr\"obner bases can be computed from lattice Gr\"obner bases, following the same
conversion 
as in the proof of Lemma \ref{lem:ineq-Markov}.  As above we may assume
that $D$ restricted to~$\LL$ has rank~$\dim(\LL)$. 
Otherwise, either no element of $\Fcin(\LL, D)$ has a minimum, or all non-empty elements of $\Fcin(\LL, D)$ have an
infinite number of minima.

\begin{lemma}
  \label{lem:computing-L,D-GBs}
  Let $D \in \zz^{r \times n}$, let $\Lcal\subseteq\Zb^{n}$ be a lattice, let $L$ be an $(n\times t)$-integer
  matrix such that the columns of $L$ are a lattice basis of~$\Lcal$, and assume that $\tilde D=DL\in\Zb^{r\times t}$
  has rank~$t$.
%
  Let $\succeq$ and $\succeq'$ be additive preorders on $\Zb^{n}$ and $\Zb^{r}$ such that for all
  $m_{1},m_{2}\in\Zb^{n}$ with $m_{1}-m_{2}\in\Lcal$,
  \begin{equation*}
    m_{1}\succeq m_{2}
    \quad\text{ if and only if }\quad
    D m_{1}\succeq' D m_{2}.
  \end{equation*}
  \begin{enumerate}
  \item %
    If $\Gcal$ is an $(\Lcal,D,\succeq)$-Gr\"obner basis, then $\Gcal'=D(\Gcal)$ is a
    $\succeq'$-Gr\"obner basis of~$\zz\tilde D$.
  \item %
    If $\Gcal'$ is an $\succeq'$-Gr\"obner basis of $\zz \tilde D$, then $D^{-1}(\Gcal')\cap\Lcal$ is an
    $(\Lcal,D,\succeq)$-Gr\"obner basis.
  \end{enumerate}
\end{lemma}

\begin{proof}
  As in the proof of Lemma~\ref{lem:ineq-Markov}, if $\Gcal=L^{-1}(\Gcal')$, then the two graphs
  \begin{equation*}
    \Fbin(\Lcal,v,D,c)_{\Gcal}
    \quad\text{ and }\quad
    \Flat(\zz \tilde{D}, Dv-c )_{\Gcal'}
  \end{equation*}
  are isomorphic as undirected graphs.  The compatibility of the preorders $\succeq$ and $\succeq'$ guarantees that the
  edge directions point to a unique sink, if a sink exists.
\end{proof}



\subsection{Sign-consistency and Graver bases}
\label{sec:Graver-bases}

Markov bases and Gr\"obner bases of lattices are related to Graver bases:
\begin{definition}
  A pair of vectors $v,v'\in\Zb^{n}$ is \emph{sign-consistent}, if $v_{i}v'_{i}\ge 0$ for all~$i=1,\dots,n$.  A sum
  $\sum_{j}v_{j}$ with $v_{j}\in\Zb^{n}$ is a \emph{conformal sum}, if any pair $v_{i},v_{i'}$ of summands is
  sign-consistent.
  
  Let $\Lcal\subseteq\Zb^{n}$ be a lattice.  An element $v\in \Lcal \setminus \{0\}$ 
  is \emph{primitive}, if the following holds: If
  $v=v_{1}+v_{2}$ is a conformal sum with $v_{1}, v_{2} \in \Lcal$ then either $v_{1}=0$ or $v_{2}=0$.  
  The set of all primitive elements is called the
  \emph{Graver basis} of~$\Lcal$.
  Alternatively, the Graver basis can be defined as the unique
  minimal subset $\Gcal_{0}\subset\Lcal$ such that any element of $\Lcal$ can be written as a conformal sum of elements
  of~$\Gcal_{0}$.
\end{definition}
Sign-consistency is an important tool to remove redundant elements from Gr\"obner bases:
\begin{lemma}
  \label{lem:conformal-redundancy}
  Let $\Gcal$ be a $\succeq$-Gr\"obner basis of a lattice~$\Lcal$.  If $v,v_{1},v_{2}\in\Gcal\setminus\{0\}$ and if $v =
  v_{1}+v_{2}$ is a conformal sum, then $\Gcal\setminus\{v\}$ is also a $\succeq$-Gr\"obner basis.
\end{lemma}

\begin{proof}
  Suppose that $u\in\Flat(\Lcal,w)$, $u+v\in\Flat(\Lcal,w)$ with $u+v\preceq u$.  Then $u+v_{1}, u+v_{2}\in\Flat(\Lcal,w)$,
  and so $\Gcal$ connects $\Flat(\Lcal,w)$. Moreover, $u+v_{1}\preceq u$ or $u+v_{2}\preceq u$ (or both).
\end{proof}

The argument in the lemma shows that the Graver basis of~$\Lcal$ is also a Gr\"obner basis for any additive
preorder.  In this sense, a Graver basis is a universal Gr\"obner basis (however, in general there may be smaller
universal Gr\"obner bases, see~\cite[Chapter 4]{Sturmfels1996:GBCP}).  In particular, any minimal
Gr\"obner basis consists of primitive vectors.

The concept of a Graver basis is tied to the coordinate hyperplanes.  Therefore, there is no natural concept of an
inequality Graver basis, or a Graver basis of a more general family of sets.  Still, Graver bases play a role when
computing Markov bases.  Namely, there are some lattices for which the Markov basis is in fact a Graver basis.  In such
cases, to compute such a basis,  it may be faster to use the program \verb|graver| instead of the program \verb|markov| (both programs belong to \verb|4ti2|).

\begin{lemma}
  \label{lem:MB-is-GB}
  If a matrix $L$ is of the form $\binom{\hat L}{-\hat L}$,
  then the Graver basis of~$\zz L$ is a minimal Markov basis of~$\zz L$.
\end{lemma}
\begin{proof}
  Recall that a lattice $\LL$ is of Lawrence type, if it consists of vectors of the form $(u,-u)$.  Any lattice of
  Lawrence type satisfies the conclusion of the lemma (e.g.~\cite[Thm.~7.1]{Sturmfels1996:GBCP}).  If $L$ has the
  indicated form, then $\zz L$ is of Lawrence type.
\end{proof}


\section{Lifting Markov and Gr\"{o}bner bases}
\label{sec:Lifting-MB-GB}

As mentioned in the previous section, Markov and Gr\"obner bases of lattices can be computed using the
software~\verb|4ti2|.  
For larger examples, the algorithms implemented in \verb|4ti2| may not terminate within a
reasonable time.  In this section we discuss an idea that allows to compute a larger Gr\"obner basis by lifting a
Gr\"obner basis that lives in lower dimensions.  For this idea to be useful, it is necessary to control both the smaller
Gr\"obner basis as well as the lifting procedure.  Later, we will apply the lifting procedure to the toric fiber
product, where the lifting procedure can be simplified using the special structure of the product.
Lifting procedures appear in special cases in 
\cite{HemmeckeMalkin09:Generating_Sets_Lattice_Ideals,KahleKroneLeykin14:Equivariant_Markov_bases,Shibuta12:GB_contraction_ideals}.

\begin{definition}
  \label{def:compatible-preorders}
  Let $\Fcal$ be a collection of subsets of $\zz^{n}$,
  let $\phi:\Zb^{n}\to\Zb^{t}$ be a linear map, and let $\succeq$ and $\succeq'$ be two additive preorders on $\Zb^{n}$ and $\Zb^{t}$.
  We say that $\succeq$ and $\succeq'$ are \emph{compatible} with $\phi$ and~$\Fcal$, if the following holds:
  \begin{itemize}
  \item For all $u,v\in\Fbf\in\Fcal$, if $\phi(u)\neq \phi(v)$, then $\phi(u)\succeq'\phi(v)$ if and only if $u\succeq v$.
  \end{itemize}
\end{definition}
In other words, $\succeq'$ determines $\succeq$ as soon as different fibers of~$\phi$ are involved.  In particular,
$\phi$ is (weakly) monotone; that is, if $u\succeq v$, then $\phi(u)\succeq'\phi(v)$.  As an example, the preorders
$\succeq_{0}$ on $\Zb^{n}$ and $\Zb^{t}$ are always compatible.

Observe the following indirect effect: If $u,u',v\in\Fbf\in\Fcal$
satisfy $\phi(u)=\phi(u')\succeq'\phi(v)\succeq'\phi(u)$ and $\phi(u)\neq\phi(v)$, then the compatibility condition implies $u\succeq v\succeq u$
and $u'\succeq v\succeq u'$, and thus $u\succeq u'\succeq u$.  Therefore, in this case, there is only one unique
preorder $\succeq$ that is compatible with~$\succeq'$.  This effect does not occur if $\succeq'$ is an order (and not
just a preorder).



If $\Fcal$ is a collection of subsets of~$\Zb^{n}$, then $\phi(\Fcal)
= \{\phi(\Fbf) : \Fbf\in\Fcal\}$ is the set of images of those subsets under the linear map~$\phi$.

\begin{definition}
  \label{def:lift}
  Let $\phi:\Zb^{n}\to\Zb^{t}$ be a linear map, let $\succeq$ be an additive preorder on~$\Zb^{n}$, and let $\Fcal$ be a
  collection of subsets of $\Zb^{n}$.  A \emph{$(\Fcal, \phi,\succeq)$-lift} of~$\Gcal\subseteq\Zb^{t}$ is a set
  $\Mcal\subset\Zb^{n}$ such that for all $\Fbf\in\Fcal$ and $v,v'\in\Fbf$ with $v\succeq v'$ that satisfy $\phi(v-v')\in\Gcal$ there are
  $m_{0}\in\Zker\phi$ 
  and $m\in\Mcal$ such that $v$, $v+m_{0}$, $v+m_{0}+m$ is a $\succeq$-non-increasing path in $\Fbf$ with
  $\phi(v+m_{0}+m)=\phi(v')$.
\end{definition}  
In other words, we can move from $v$ to $v'$ by applying first a move $m_{0}$ from the kernel of~$\phi$, then a move $m$
from the lift, and finally again a move $v'-(v+m_{0}+m)$ from the kernel of~$\phi$.  Apart from the last step, all other
steps should be non-increasing.
  If $\Fcal$ and $\succeq$ are understood from the context, we simply speak of a \emph{$\phi$-lift}.

\begin{figure}
  \centering
  \begin{tikzpicture}
    \foreach \i in {1,...,4} { \foreach \j in {1,...,4} { \path (\i,\j) coordinate (X\i\j); } }
    \begin{scope}[lightgray] 
      \draw (X11) -- (X12) -- (X13) -- (X14); \draw (X21) -- (X22) -- (X23) -- (X24); \draw (X31) -- (X32) -- (X33) --
      (X34); \draw (X41) -- (X42) -- (X43) -- (X44);
    \end{scope}
    \begin{scope}[lightgray] 
      \draw (X13) -- (X22) -- (X31);
    \end{scope}
    \foreach \i in {1,...,4} { \foreach \j in {1,...,4} { \fill (X\i\j) circle (3pt); } }
    \draw (X21) edge [bend right=30,red,->,>=triangle 60,shorten >=3pt] (X41);
    \path[->,>=stealth] (2.5,0.5) edge node[left] {$\phi$} (2.5,-0.5);
    \foreach \i in {1,...,4} { \path (\i,-1) coordinate (Y\i); }
    \draw[lightgray] (2,-1) edge (1,-1) edge (3,-1); 
    \foreach \i in {1,...,4} { \fill (Y\i) circle (3pt); }
    \draw (2,-1) edge[bend right=30,red,->,>=triangle 60,shorten >=3pt] node[below] {$g$} (4,-1); 
    \begin{scope}[red]
      \draw (X23) edge[bend left=20,->,>=triangle 60,shorten >=3pt] node[right] {$m_{0}$} (X21);
      \node at (3,0.5) {$m$};
      \draw (X44) edge[bend left=20,->,>=triangle 60,shorten >=3pt] (X41);
      \draw[dashed] (X23) edge (X44);
      \fill (X23) circle (2pt);
      \fill (X44) circle (2pt);
      \node at (1.7,3) {$v$};
      \node at (4.4,4) {$v'$};
      \fill (Y2) circle (2pt);
      \fill (Y4) circle (2pt);
    \end{scope}
  \end{tikzpicture}
  
  \caption{An illustration of the definition of a $\phi$-lift.  The arrows point in the direction of $\succeq$-smaller nodes.}
  \label{fig:lifting}
\end{figure}
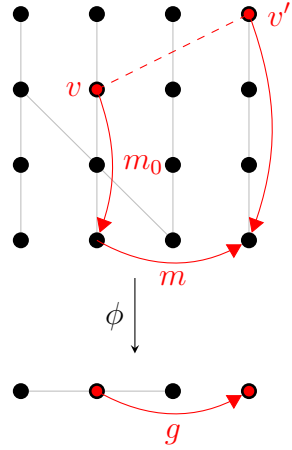
Figure~\ref{fig:lifting} illustrates the definition.  We will later see that lifts exist in the setting of Gr\"obner
bases of lattices (Section~\ref{sec:lifting-GBs-algo}).

Lifting allows us to relate certain Gr\"obner bases in $\Zb^{n}$ and~$\Zb^{t}$.  In order to state this correspondence
precisely, the following definitions are needed:

\begin{definition}
  Let $\Fcal$ be a family of subsets of $\Zb^{n}$, and let $\phi:\Zb^{n}\to\Zb^{t}$ be a linear map.
  Let $\succeq$ and $\succeq'$ be additive preorders on $\Zb^{n}$ and $\Zb^{t}$, respectively. 
  A $\succeq'$-Gr\"obner basis of $\phi(\Fcal)$ is called a \emph{projected fiber Gr\"obner basis} (\emph{PF Gr\"obner basis}).
  A \emph{kernel Gr\"obner basis} is a $\succeq$-Gr\"obner basis of the family of sets of the form
  \begin{equation*}
    \Fbf  \cap (u + \ker \phi), \quad\text{ for $\Fbf\in\Fcal$ and $u\in\Zb^{n}$;}
  \end{equation*}
  that is, the fibers of $\phi$ restricted to some $\Fbf\in\Fcal$.
\end{definition}

\begin{theorem}
  \label{thm:GB-lifting}
  Let $\Fcal$ be a family of subsets of $\Zb^{n}$, and let $\phi:\Zb^{n}\to\Zb^{t}$ be a linear map.
  Let $\succeq$ and $\succeq'$ be additive preorders on $\Zb^{n}$ and $\Zb^{t}$, respectively, that are compatible with~$\phi$. 
  Let $\Gcal$ be a 
  PF Gr\"obner basis, and let $\Mcal_{1}$ be a $(\Fcal, \phi,\succeq)$-lift of $\Gcal$.
  Let $\Mcal_{0}$ be a kernel Gr\"obner basis.
  Then $\Mcal_{0}\cup\Mcal_{1}$ is a $\succeq$-Gr\"obner basis of~$\Fcal$.
\end{theorem}

\begin{proof}
  Let $\Mcal:=\Mcal_{0}\cup\Mcal_{1}$.  We want to apply Lemma~\ref{lem:Markov-is-Grb}.  First, we show that
  $\Mcal$ is a Markov basis.  Let $u,v\in\Fbf\in\Fcal$.  Since $\Gcal$ is a $\succeq'$-Gr\"obner basis for
  $\phi(\Fcal)$, there are $g_{1},\dots,g_{r} \in\pm\Gcal$ such that $\phi(u)$, $\phi(u) + g_{1}$, \dots, $\phi(u) + g_{1}
  + \dots + g_{r}=\phi(v)$ is a path in~$\phi(\Fbf)$ from $\phi(u)$ to~$\phi(v)$.  Using Definition~\ref{def:lift},
  this path lifts to a path in~$\Fbf$ from $u$ to~$v$ with moves in~$\pm\Mcal$.

  Next, we show that the two conditions of Lemma~\ref{lem:Markov-is-Grb} are satisfied.  The argument is similar, but
  now we need to take into account the preorders.
  For the first condition, let $u,v\in\Fbf$ be $\succeq$-minimal in~$\Fbf$.  Since $\phi$ is monotone,
  $\phi(u),\phi(v)$ are $\succeq'$-minimal in~$\phi(\Fbf)$.  Since $\Gcal$ is a
  $\succeq'$-Gr\"obner basis for $\phi(\Fcal)$, there are $g_{1},\dots,g_{r} \in\pm\Gcal$ such that $\phi(u)$, $\phi(u) +
  g_{1}$, \dots, $\phi(u) + g_{1} + \dots + g_{r}$ is a non-increasing (with respect to $\succeq'$) path in~$\phi(\Fbf)$
  with $\phi(u) + g_{1} + \dots + g_{r} = \phi(v)$.  Since $\Mcal_{1}$ is a $(\Fcal, \phi, \succeq)$-lift,
  we can use moves from $\Mcal_{0}$ and $\Mcal_{1}$ to lift this path to a non-increasing path in~$\Fbf$
  as follows: In the first step, since $\phi(u)+g_{1}\in\phi(\Fbf)$, there exists $\tilde u_{1}\in\Fbf$ with
  $\phi(\tilde u_{1})=\phi(u)+g_{1}$.  Let $m,m_{0}$ be as in Definition~\ref{def:lift} applied to~$u,\tilde u_{1}$ in
  place of~$v,v'$.  Then there is a non-increasing path from $u$ to $u+m_{0}$ using moves from~$\Mcal_{0}$.  Adding the
  move~$m\in\Mcal_{1}$, we obtain a non-increasing path from~$u$ to $u_{1}:=u+m_{0}+m\in\Fbf$, where $u_{1}$ satisfies
  $\phi(u_{1})=u+g_{1}$.  Iterating this procedure, we obtain a non-increasing path in~$\Fbf$ with edges in~$\Mcal$ that
  starts in $u$ and ends in some~$u_{r}$, with $\phi(u_{r})=\phi(v)$.  Since the path is non-increasing, $u_{r}$ is also
  $\succeq$-minimal.  Therefore, we can concatenate a non-increasing path from $u_{r}$ to~$v$ using moves
  in~$\Mcal_{0}$.  This shows the first condition.

  For the second condition, assume that there is $v\in\Fbf$ with~$v\prec u$.  If $\phi(u)$ is $\succeq'$-minimal
  in~$\phi(\Fbf)$, then so is $\phi(v)$, and there is a $\succeq'$-non-increasing path from $\phi(u)$ to~$\phi(v)$.  As
  above, this path lifts to a non-increasing path from $u$ to some~$u'$ with $\phi(u')=\phi(v)$.  If $u'\prec u$, then
  we are done.  Otherwise, $v\prec u'$, and so there is a non-increasing path in $\Fbf\cap(v+\ker_{\Zb}\phi)$ with moves
  in $\Mcal_{0}$ from~$u'$ to some~$u''$ with~$u''\prec u'$, since $\Mcal_{0}$
  is a kernel Gr\"obner basis.   Joining these two non-increasing paths proves the
  statement.  If $\phi(u)$ is not $\succeq'$-minimal, then 
there are
  $g_{1},\dots,g_{r} \in \Gcal$ such that $\phi(u)$, $\phi(u) + g_{1}$, \dots, $\phi(u) + g_{1} + \dots + g_{r}$ is a
  non-increasing (with respect to $\succeq'$) path in~$\phi(\Fbf)$ with $\phi(u) + g_{1} + \dots + g_{r} \prec \phi(u)$.
  Again, this path can be lifted and yields a non-increasing path in~$\Fbf$ with edges in~$\Mcal$ that starts in~$u$ and
  ends in a point~$u'$ with $\phi(u')=\phi(u)+g_{1}+\dots+g_{r}\prec\phi(u)$.  Since $\succeq$ and $\succeq'$ are
  compatible, $u'\prec u$.
\end{proof}



In some instances we will encounter later, it can be more straightforward to check the following
more demanding condition than  $\Mcal$ being a lift of $\Gcal$.

\begin{lemma}
  Let $\Gcal\subseteq\Zb^{t}$, let $\Lcal\subseteq\Zb^{n}$ be a lattice, and let $\Mcal\subseteq\Zb^{n}$.  Assume
  that for any $g\in\Gcal$ and $m\in\Lcal$ with $\phi(m)=g$, there is a sign-consistent
  decomposition $m=m_{0}+m_{1}$ with $m_{1}\in\pm\Mcal$ and $\phi(m_{0})=0$.
  Then $\Mcal$ is a $(\Fclat(\Lcal),\phi,\succeq_{0})$-lift of~$\Gcal$.
\end{lemma}

\begin{proof}
  Let $v,v'\in\Flat(\Lcal,v)$ with $\phi(v-v')=g$, and decompose $m=v'-v$ as in the statement of the lemma.  The
  sign-consistency condition implies that $v+m_{0}\in\Flat(\Lcal,v)$, and so $m=m_{0}+m_{1}+0$ is a decomposition of
  $v'-v$ as in the definition of a lift.
\end{proof}

\bigskip%
To apply Theorem~\ref{thm:GB-lifting} to compute a Gr\"obner basis of
$\Fcal$, the following needs to be done:
\begin{enumerate}
\item Compute a kernel Gr\"obner basis $\Mcal_{0}$.
\item Compute a PF Gr\"obner basis $\Gcal$ of $\phi(\Fcal)$.
\item Compute a lift $\Mcal_{1}$ of $\Gcal$.
\end{enumerate}
We discuss these three points in the special case $\Fcal=\Fcal(\Bcal)$ 
for some integer matrix~$\Bcal$. 

The first point is the easiest: In fact,
in this context, a kernel Gr\"obner basis is nothing but a Gr\"obner basis of the lattice $\Zker\Bcal\cap\Zker\phi$.
The lattice $\Zker\Bcal\cap\Zker\phi$ can also be described as the integer kernel of the matrix $\Bcal^{\phi}$ with columns
\begin{equation*}
  \binom{b_{i}}{\phi(e_{i})},
  \qquad
  \text{ where $b_{i}$ denotes the $i$th column of $\Bcal$ and $e_{i}$ the $i$th unit vector.}
\end{equation*}

Before discussing the other two points, let us give another interpretation to 
$\Bcal^{\phi}$.
%
Given $\Bcal$ and $\phi$ as above, let $\phi'$ be the linear map corresponding to~$\Bcal^{\phi}$.  In the following we
only care about how $\phi$ acts on each fiber.  Now, lifting along $\phi$ is essentially the same as lifting
along~$\phi'$, since both maps have the same kernel Gr\"obner bases, and the projected fiber Gr\"obner bases are
equivalent.  Moreover, the linear map $\psi_{\Bcal}$ corresponding to $\Bcal$ factorizes through~$\phi'$ (to be precise,
$\psi_{\Bcal}$ arises from $\phi'$ by composition with a coordinate projection).  Therefore, in our study of lifting, we
could restrict attention to linear maps $\phi:\Zb^{n}\to\Zb^{t}$ that are factors of~$\psi_{\Bcal}$,
and in this case, $\Bcal^{\phi}$ is nothing but a matrix that represents~$\phi$.


\subsection{Gr\"obner bases of projected fibers}
\label{sec:proj-GBs}

Let $u\in\phi(\Fbf(\Bcal, b))$, and let $v\in\Fbf(\Bcal, b)$ such that $u=\phi(v)$.  Then $\Bcal^{\phi} v =
\binom{b}{u}$.  Conversely, if $\binom{b}{u}$ lies in the affine semigroup~$\nn\Bcal^{\phi}$,
then $u$ lies in $\phi(\Fbf(\Bcal, b))$.  In other words, descriptions of the projected fibers $\phi(\Fbf(\Bcal, b))$
can be obtained from suitable descriptions of~$\nn\Bcal^{\phi}$.

Let $N\Bcal^{\phi} := (\Zb\Bcal^{\phi}\cap\Rb_{\ge}\Bcal^{\phi})$ be the \emph{normalization} of~$\Nb\Bcal^{\phi}$.
Elements of $N\Bcal^{\phi}\setminus\Nb\Bcal^{\phi}$ are called \emph{holes}.  The semigroup $\Nb\Bcal^{\phi}$ is
\emph{normal} if and only if $\Nb\Bcal^{\phi} = N\Bcal^{\phi}$; 
that is, if and only if there are no holes.
Normality of semigroups can be checked using the software~\verb|Normaliz|\cite{Normaliz}.
See~\cite{HemmeckeTakemureYoshida09:Computing_holes_in_semigroups} for an algorithm to compute the holes of non-normal semigroups.

\begin{lemma}
  \label{lem:normal-noholes}
  Let $\Bcal \in \zz^{h \times n}$, let $\phi:  \zz^{n}  \rightarrow \zz^{t}$ be a linear map, and let
  $$
  \LL  = \Big\{ u \in \zz^{t}: \binom{0}{u} \in \zz \Bcal^{\phi} \Big\}.
  $$
  If $\Nb\Bcal^{\phi}$ 
  is normal, then there exists 
  an $r \times t$ integer matrix $D$ such that the following holds: For any $b\in\Nb\Bcal$, there exists $c\in\Zb^{r}$  with
  \begin{equation*}
    \phi(\Fbf(\Bcal, b)) = \{ u\in \LL + v : D u \ge c \} = \Flat(\Lcal,v,D,c).
  \end{equation*}
  The matrix $D$ can be obtained from an inequality description of the cone~$\Rb_{\ge}\Bcal^{\phi}$.
\end{lemma}
\begin{proof}
  If $\Nb\Bcal^{\phi}$ is normal, then it is equal to the intersection of the lattice $\Zb\Bcal^{\phi}$ and
  the polyhedral cone $\Rb_{\ge}\Bcal^{\phi}$.
Let $(D_{1}\;D_{2})$ be a matrix such that
$$
\Rb_{\ge}\Bcal^{\phi}  =   \Big\{  (b,u)  \in \rr^{n + t}  : D_{1}b + D_{2}u  \geq   0 \Big\}.
$$  
Hence, if $\Fbf(\Bcal,b)\neq\emptyset$, then $\phi(\Fbf(\Bcal, b))  =  \big\{ u  \in \LL + v  :   D_{2}u  \geq  -D_{1} b  \big\}$
where $v \in \zz^{t}$ is any vector such that $\binom{b}{v} \in \zz \Bcal^{\phi}$.  
\end{proof}

By Lemma~\ref{lem:normal-noholes}, if $\nn \Bcal^{\phi}$ is normal,
a PF-Gr\"obner basis can be computed via an $(\LL, D)$-Gr\"obner basis for suitable
$\LL$ and~$D$.  
We demonstrate this in the next example.
In general, the $(\LL, D)$-Gr\"obner basis might be larger
than a minimal PF-Gr\"obner basis, because a PF-Gr\"obner basis does not need to work for all
sets of the form $\Fbin(\LL,v,D, c)$ for all $c \in \zz^{r}$, but it suffices if it works for those fibers
where $c$ lies in 
the affine semigroup $-\Nb D_{1}\Bcal$.

\begin{example}
  \label{ex:simple-PFs}
  Let
  $\Bcal = \begin{smallpmatrix}
    1 & 1 & 1 & 1 \\
    0 & 0 & 1 & 2
  \end{smallpmatrix}$ and
  $\phi = 
  \begin{smallpmatrix}
    1 & 0 & 0 & 0 \\
    0 & 1 & 1 & 0 \\
    0 & 0 & 0 & 1
  \end{smallpmatrix}$.
  Then
  $\Bcal^{\phi} =
    \begin{smallpmatrix}
      1 & 1 & 1 & 1 \\
      0 & 0 & 1 & 2 \\
      1 & 0 & 0 & 0 \\
      0 & 1 & 1 & 0 \\
      0 & 0 & 0 & 1
    \end{smallpmatrix}$.
  This matrix has rank four, and hence the kernel Markov basis is empty.

  Denote the coordinates in $\Rb^{5}$ by $x_{1},x_{2},y_{1},y_{2},y_{3}$.  
  According to \verb|Normaliz|, the affine semigroup $\Nb\Bcal^{\phi}$ is normal and consists of all integer solutions of
  \begin{gather*}
    y_{1} + y_{2} + y_{3} = x_{1},
    \\
    y_{1} \ge 0,
    \qquad
    y_{1} + y_{2} \le x_{1},
    \qquad
    y_{1} + y_{2} \ge x_{1} - \frac12 x_{2},
    \qquad
    2 y_{1} + y_{2} \le 2 x_{1} - x_{2}.
  \end{gather*}
  A Markov basis for these projected fibers is the same as a Markov basis in Example~\ref{ex:simple-ilMB}.  In fact, the
  gray set in Figure~\ref{fig:simple-iMB} is equal to the projected fiber $\phi(\Fbf(\Bcal, \binom{4}{5}))$.  \qed
\end{example}

Even if $\Nb\Bcal^{\phi}$ is not normal, the inequality description of $\Rb_{\ge}\Bcal^{\phi}$ gives valuable
information about~$\Nb\Bcal^{\phi}$.  Namely, $\Nb\Bcal^{\phi}$ can be described as
$\Nb\Bcal^{\phi}=N\Bcal^{\phi}\setminus H$, where $H$ denotes the set of holes of $\Nb\Bcal^{\phi}$.  A similar
description can be given to the projected fibers: If $(b,h)\in H$, then we call
$h\in\Nb^{t}$ a \emph{hole} of $\phi(\Fbf(\Bcal, b))$.  In some instances, the set of holes is small enough that we can
still find a good PF Markov or Gr\"obner basis.  We will illustrate this in Section \ref{sec:K4-e}.
%


\subsection{Lifting Gr\"obner bases of lattices}
\label{sec:lifting-GBs-algo}

Finally, we give an algorithm for lifting for Gr\"obner bases of a lattice~$\Lcal$; that is, we want to compute a
Gr\"obner basis of~$\Lcal$ from a PF Gr\"obner basis~$\Gcal$.  Since a union of lifts of singletons $\{g\}$ for all
$g\in\Gcal$ is a lift of $\Gcal$, it suffices to know how to lift a single element.
Lifting is easy if $\Zker\phi\cap\Lcal=\{0\}$.  In this case, the lift of $g$ consists of the unique element
$m\in\Lcal\cap\phi^{-1}(g)$.  This special case of lifting appears
in~\cite{KahleKroneLeykin14:Equivariant_Markov_bases}.  In general, the problem to lift $g\in\Gcal$ can again be formulated
as a Gr\"obner basis computation:

For any $\Fbf\in\Fcal$ and $u_{1},u_{2}\in\phi(\Fbf)$ with $u_{2}-u_{1}=g$ let
\begin{equation*}
  \Fblift(\Fbf, \phi,  u_{1},u_{2}) := \big\{ v\in\Fbf : \phi(v)\in\{u_{1}, u_{2}\} \big\}
  = \big\{ v\in\Fbf : \phi(v)-u_{1}\in\{0,g\} \big\}.
\end{equation*}
If $\hat\Mcal$ is a Gr\"obner basis of the family
\begin{equation*}
\Fclift(\Fcal, \phi, g) :=
  \big\{\Fblift(\Fbf, \phi,  u_{1},u_{2})  : \Fbf\in\Fcal,u_{1},u_{2}\in\phi(\Fbf),u_{2}-u_{1}=g\big\},
\end{equation*}
then $\Mcal_{g}=\{m\in\hat\Mcal:\phi(m)= g\}$ lifts~$g$.

\begin{lemma}\label{lem:lift-to-ineq}
  Let $\LL \subseteq \zz^{n}$ be a lattice and $D$ an $r \times n$ integer matrix.  Then
$$
\Fclift(\Fcin(\LL, D), \phi, g)  \subseteq  \Fcin(\LL_g, D_g)
$$
for a suitable lattice $\LL_{g}$ and matrix $D_{g}$.
\end{lemma}

\begin{proof}
  Let $d_{g}$ be the linear form on $\Zb^{t}$ defined by $d_{g}(h) = \langle g,\phi(h)\rangle$, and consider the lattice
  $\Lcal_{g} = \phi^{-1}(\Zb g) \cap \LL$.  
  Note that for the lifting problem to yield any nontrivial moves, we
  need only look at fibers $\Fbin(\LL,v,D,c)$
  such that there are $u_{1}, u_{2}\in \phi(\Fbin(\LL,v,D,c))$ with $u_{2}-u_{1}=g$.
  In particular, we can restrict to $v$ such that $\phi(v) = u_{1}$.
  For such $v$, 
  \begin{align*}
    \Fblift(\Fbin(\LL,v,D,c), \phi, u_{1},u_{2}) & =
    \{ w \in \Fbin(\LL,v,D,c) : \phi(w)\in\{u_{1}, u_{2}\} \}  \\
    &= \{ w \in \LL + v :  Dw \ge c, \phi(w) \in \{u_{1}, u_{2} \} \}  \\
    &= \{ w \in \LL_{g} + v : Dw \geq c, \langle g,u_{1}\rangle \leq d_{g}(w) \leq \langle g, u_{2}\rangle \}.
  \end{align*}
  The last equality can be seen as follows: If $w\in\LL+v$ and $\phi(w)\in\{u_{1}, u_{2} \}$, then
  $\phi(w-v)\in\{0,g\}$, and so $w\in\LL_{g}+v$.  Moreover, $\langle g,u_{1}\rangle \leq d_{g}(w) \leq \langle g,
  u_{2}\rangle$.  Conversely, if $w\in\LL_{g}+v$, then $\phi(w)\in u_{1}+\Zb g$.  The inequality $\langle
  g,u_{1}\rangle \leq d_{g}(w) \leq \langle g, u_{2}\rangle$ implies $\phi(w)=u_{1}$ or $\phi(w)=u_{1}+g=u_{2}$.  Hence,
  the statement follows with our choice of $\LL_{g}$ and with $D_{g}$ the matrix $D$ with two rows appended
  corresponding to $d_{g}$ and~$-d_{g}$.
\end{proof}

In many situations, Lemma~\ref{lem:lift-to-ineq} allows us to calculate lifts
using inequality Gr\"obner bases.  
The following proposition follows directly:

\begin{proposition}
  \label{prop:lifting-procedure-GB}
  Let $\succeq,\succeq'$ be compatible additive preorders for $\phi:\Zb^{n}\to\Zb^{t}$, let $\Lcal\subseteq\Zb^{n}$ be a
  lattice, $D\in\Zb^{r\times n}$, and let $\Gcal$ be a PF Gr\"obner basis for~$\Fcin(\Lcal,D)$.  For each~$g\in\Gcal$, let
  $\Lcal_{g}$, $D_{g}$ be as in Lemma~\ref{lem:lift-to-ineq}, let $\Mcal_{g}'$ be an
  $(\Lcal_{g},D_{g},\succeq)$-Gr\"obner basis, and let $\Mcal_{g} = \{ m\in\Mcal_{g}' : \phi(m) = + g \}$.  Then
  $\bigcup_{g\in\Gcal}\Mcal_{g}$ is an $(\Fcin(\LL,D),\phi,\succeq)$-lift of~$\Gcal$.
\end{proposition}

\begin{example}
  \label{ex:simple-lifting}
  We continue Example~\ref{ex:simple-PFs}, using the Markov basis~\eqref{eq:simple-ilMB}.  In this case, since
  $\ker\Bcal^{\phi}=\{0\}$, the lifting procedure yields one lift for each of the two vectors.  Hence the lifted Markov
  basis is
  \begin{equation*}
    \Mcal = \Big\{ (1, -1, 0, 0), \quad (1, 0, -2, 1) \Big\}.
  \end{equation*}
  This is also the Markov basis that \verb|4ti2| computes when given the matrix~$\Bcal$.  \qed
\end{example}

Less trivial examples of lifting appear in Section~\ref{sec:models}.

\subsection{The codimension-one case and the slow-varying property}
\label{sec:slow-varying}

The complexity of projecting the fibers 
and of lifting 
crucially depends on the choice of the map~$\phi$.  How to find a good~$\phi$ is difficult to say in general.  One
aspect is the dimensionality of the projected fibers.
\begin{definition}
  The \emph{codimension} of the $(\Fcal,\phi,\succeq)$-lifting is defined as $\sup_{\Fbf\in\Fcal}\dim(\phi(\Fbf))$, where
  $\dim(\phi(\Fbf))$ denotes the dimension of the affine hull of~$\phi(\Fbf)$.
\end{definition}
In the case $\Fcal=\Fcal(\Bcal)$ of matrices, the codimension is $\dim(\phi(\Zker\Bcal))$, in the case
$\Fcal=\Fclat(\Lcal)$ of lattices, the codimension is~$\dim(\phi(\Lcal))$, where in both cases $\dim$ denotes the
dimension of a lattice.

In this section we focus on the codimension-one case and relate our theory to some results
of~\cite{EngstromKahleSullivant13:TFP-II}.  In this case, the
projected fibers are at most one-dimensional.  Let $g\in\Zb^{t}$ be a generator of~$\phi(\Zker\Bcal)$.    For any $b\in\Nb\Bcal$ and $u_{0}\in\Fbf(\Bcal,b)$ we have
$\phi(\Fbf(\Bcal,b)) \subseteq u_{0} + \Zb g$.  If there are no holes, then $\phi(\Fbf(\Bcal,b))$ consists of
consecutive elements of $u_{0} + \Zb g$; that is $\phi(\Fbf(\Bcal,b)) = \{u_{0} + k g : l \le k \le l' \}$ for some
$l,l'\in\Zb$.  In this case, $\{\pm g\}$ is a PF Gr\"obner basis for any additive preorder on~$\Zb^{t}$.

\begin{definition}
  In the codimension-one case, a Gr\"obner basis $\Mcal$ of~$\Bcal$ is \emph{slow-varying} with respect to~$\phi$, if there
  exists a single vector~$g\in\Zb^{t}$ such that $\phi(\Mcal)\subseteq\{0,\pm g\}$.
\end{definition}
Slow-varying Markov bases are useful in the gluing procedure in the toric
fiber product construction as shown in \cite{EngstromKahleSullivant13:TFP-II}.
Clearly, a slow-varying Gr\"obner basis exists if and only if $\{\pm g\}$ is a PF Gr\"obner basis for any additive
preorder on~$\Zb^{t}$.  Hence:
\begin{lemma}
  \label{lem:slow-varying-lift}
  Assume that $\phi$ has codimension one with respect to~$\Bcal$.
  If $\Nb\Bcal^{\phi}$ is normal, then there exists a slow-varying Gr\"obner basis for any additive preorder.
\end{lemma}


\section{The toric fiber product}
\label{sec:TFP}

We now turn our attention to the toric fiber product construction.  This construction involves several maps that lend
themselves naturally as candidates for lifting.  In Sections~\ref{sec:kernel-Grobner-tfp} to~\ref{sec:lifting-TFP}, we
show how the results of the previous section help to compute Gr\"obner bases of toric fiber products.
Section~\ref{sec:simple-ex} contains an elaborate example.

We first recall the construction and fix the notation.
The toric fiber product is defined for general $\nn \Acal$-homogeneous ideals in~\cite{Sullivant07:TFPs}.  We focus
exclusively on the case of toric fiber products of toric ideals, and hence, toric fiber products of matrices.

\begin{definition}
  \label{def:TFP}
  Let $\Acal\in\Zb^{s\times t}$ be an integer matrix with columns~$a_{1},\dots,a_{t}$.
  Any surjection $\phi:[n]\to[t]$ induces a surjective map $\Zb^{n}\to\Zb^{t}, e_{i}\mapsto e_{\phi(i)}$, which we
  denote by $\phi$ again.  Let $\Bcal=(b_{1},\dots,b_{n})$ be an integer matrix with $n$ columns.
  We say that $\Bcal$ is \emph{$\Acal$-graded} by $\phi$, if one of the following two equivalent statements is
  satisfied:
  \begin{itemize}
  \item
    There is a linear map $\pi:\Nb\Bcal\to\Nb\Acal$ with $\pi(b_{i})=a_{\phi(i)}$.
  \item The map $\phi:\Zb^{n}\to\Zb^{t}$ satisfies $\phi(\Zker\Bcal)\subseteq\Zker\Acal$.
  \end{itemize}

  Given two matrices $\Bcal,\Bcal'$ that are $\Acal$-graded by two maps $\phi,\phi'$, the \emph{toric fiber product} is
  the matrix 
  \begin{equation*}
    \Bcal\TFP\Bcal' :=
    \left\{
      \begin{pmatrix}
        b_{i} \\ b'_{j}
      \end{pmatrix}
      :  \phi(i)=\phi'(j)
    \right\}
  \end{equation*}
  that consists of all pairs of columns from $\Bcal$ and $\Bcal'$ that are mapped to the same column of~$\Acal$.  The
  \emph{codimension} of this toric fiber product is equal to~$\dim\Zker\Acal$.
\end{definition}
Lemma~\ref{lem:codim-TFP-codim-liftings} below relates the codimension of a toric fiber product to the codimensions of
natural associated liftings.

Consider the map $\psi:\Zb^{\Bcal\TFP\Bcal'}\to\Zb^{\Bcal}$ that maps the unit vector $e_{i,j}$ corresponding to
$(b_{i},b'_{j})$ to the $i$th unit vector $e_{i}\in\Zb^{\Bcal}$, and consider the corresponding map
$\psi':\Zb^{\Bcal\TFP\Bcal'}\to\Zb^{\Bcal'}$ that maps $e_{i,j}$ to $e_{j}\in\Zb^{\Bcal'}$.  Then the following
diagram commutes:
\begin{center}
  \begin{tikzpicture} 
    \node (B1xAB2) at (2,4) {$\Zb^{\Bcal\TFP\Bcal'}$};
    \node (B1) at (0,2) {$\Zb^{\Bcal}$};
    \node (B2) at (4,2) {$\Zb^{\Bcal'}$};
    \node (A) at (2,0) {$\Zb^{\Acal}$};
    \begin{scope}[->,>=stealth]
      \path (B1xAB2) edge node[above] {$\psi$} (B1);
      \path (B1xAB2) edge node[above] {$\psi'$} (B2);
      \path (B1) edge node[above] {$\phi$} (A);
      \path (B2) edge node[above] {$\phi'$} (A);
      \path (B1xAB2) edge node[auto] {$\xi$} (A);
    \end{scope}
  \end{tikzpicture}
\end{center}
where $\xi = \phi \circ \psi = \phi' \circ \psi'$.

Let $\succeq_{\times}$, $\succeq_{\Bcal}$, $\succeq_{\Bcal'}$ and $\succeq_{\Acal}$ be
additive preorders on $\Zb^{\Bcal\TFP\Bcal'}$, $\Zb^\Bcal$, $\Zb^{\Bcal'}$, and $\Zb^\Acal$, respectively,
that are compatible with $\phi$, $\phi'$ and~$\xi$.  In general, it is not possible to require
that $\psi$ and $\psi'$ are also compatible with respect to these orders.  Instead, we call $\succeq_{\times}$
\emph{compatible}, if it satisfies the following weaker property:
\begin{itemize}
\item For any $u_{1},u_{2}\in\Zb^{\Bcal\TFP\Bcal'}$, if $\psi(u_{1})\succeq_{\Bcal}\psi(u_{2})$ and if
  $\psi'(u_{1})\succeq_{\Bcal'}\psi'(u_{2})$, then $u_{1}\succeq_{\times}u_{2}$.
\end{itemize}
For given preorders $\succeq_{\Bcal}$, $\succeq_{\Bcal'}$ and $\succeq_{\Acal}$, a compatible
preorder $\succeq_{\times}$ on $\Zb^{\Bcal\TFP\Bcal'}$ can be constructed as follows:
\begin{align*}
  u_{1}\succeq_{\times}u_{2}
  \quad:\Longleftrightarrow\quad
  & \psi(u_{1})\succ_{\Bcal}\psi(u_{2}) \\
  & \text{or } \psi(u_{1})\succeq_{\Bcal}\psi(u_{2})\succeq_{\Bcal}\psi(u_{1})\text{ and }\psi'(u_{1})\succeq_{\Bcal'}\psi'(u_{2}).
\end{align*}

Our goal is to compute $\succeq_{\times}$-Gr\"obner bases of $\ker_\zz {\Bcal\TFP\Bcal'}$.  We want to apply the lifting
machinery from the previous section and lift along the map~$\xi$.
To apply Theorem \ref{thm:GB-lifting}, we need to understand the kernel Gr\"obner basis, the PF Gr\"obner basis, and we
need to lift the PF Gr\"obner basis.  This will be described in the next three sections.
A key result is that we only need to compute 
lifts along $\phi$ and~$\phi'$, which can be ``glued'' to produce lifts along~$\xi$.  The complexity of these lifts is
governed by the codimension of the toric fiber product.
\begin{lemma}
  \label{lem:codim-TFP-codim-liftings}
  The codimension of the toric fiber product $\Bcal\TFP\Bcal'$ is not less than the codimension of
  $(\Fbf(\Bcal\TFP\Bcal'),\xi,\succeq_{\times})$-liftings, $(\Fbf(\Bcal),\phi,\succeq_{\Bcal})$-liftings and
  $(\Fbf(\Bcal'),\phi',\succeq_{\Bcal'})$-liftings.
\end{lemma}
\begin{proof}
  This follows from the inclusion $\xi(\Zker(\Bcal\TFP\Bcal'))\subseteq\phi(\Zker(\Bcal))\subseteq\Zker(\Acal)$, together with the
  symmetric inclusion.
\end{proof}

The results in this section are very technical.  A simple example is given in Section~\ref{sec:simple-ex}, after
presenting the theory.  Larger examples that show how to apply the results of this section to hierarchical models will
be given in Section~\ref{sec:models}.


\subsection{Kernel Gr\"obner basis and the associated codimension zero toric fiber product}
\label{sec:kernel-Grobner-tfp}

To compute the kernel Gr\"obner basis, we need the following definition:

\begin{definition}
  Let $\Bcal,\Bcal'$ be integer matrices that are $\Acal$-graded via maps $\phi,\phi'$ as above.  The
  \emph{associated codimension zero toric fiber product} is the matrix $\Bcal^{\phi}\TFP[\tilde\Acal](\Bcal')^{\phi'}$,
  where $\tilde\Acal$ is the unit matrix in~$\Nb^{t\times t}$ and where $\Bcal^{\phi}$ and $(\Bcal')^{\phi'}$ are
  $\Acal$-graded using the same maps~$\phi,\phi'$.
\end{definition}
\begin{lemma}
  \label{lem:cod-zero-connects-fibers}
  $\Zker(\xi)\cap\Zker(\Bcal\TFP\Bcal') = \Zker(\Bcal^{\phi}\TFP[\tilde\Acal](\Bcal')^{\phi'})$.  
  Hence, when lifting along~$\xi$, a kernel Gr\"obner basis is given by a Gr\"obner basis of
  $\Zker(\Bcal^{\phi}\TFP[\tilde\Acal](\Bcal')^{\phi'})$.
\end{lemma}

\begin{proof}
  Observe that $\Bcal\TFP\Bcal'$  can be identified with a submatrix of $\Bcal^{\phi}\TFP[\tilde\Acal](\Bcal')^{\phi'}$.
  In fact, a sequence of row operations turns the matrix $\Bcal^{\phi}\TFP[\tilde\Acal](\Bcal')^{\phi'}$ into the matrix with columns
  \begin{equation*}
    \begin{pmatrix}
      b_{i} \\
      b'_{j} \\
      e_{\phi(i)}
    \end{pmatrix}
    \text{ for all }i,j
    \text{ with }\phi(i) = \phi'(j).
  \end{equation*}
  Clearly, the kernel of this last matrix is $\Zker(\xi)\cap\Zker(\Bcal\TFP\Bcal')$.
\end{proof}

Note that $\ker\tilde\Acal=\{0\}$, and so  $\Bcal^{\phi}\TFP[\tilde\Acal](\Bcal')^{\phi'}$ is a  
codimension zero toric fiber product.  Computation of Markov bases and Gr\"obner bases
of codimension zero toric fiber product was described in \cite{Sullivant07:TFPs}.
We review the main result here.

Let $m\in\Zker\Bcal^\phi$.  Then $\phi(m)\in\Zker(\tilde\Acal)=\{0\}$, and so $\phi(m^{+})=\phi(m^{-})$.  Hence there exist maps
$\sigma_{+},\sigma_{-}$ such that $m = \sum_{i}e_{\sigma_{+}(i)} - \sum_{i}e_{\sigma_{-}(i)}$ and
$\phi(e_{\sigma_{+}(i)})=\phi(e_{\sigma_{-}(i)})$.  Choose a map $\tau$ with
$\phi'(e_{\tau(i)})=\phi(e_{\sigma_{\pm}(i)})$.  Then
\begin{equation*}
  \tilde m = \sum_{i}e_{\sigma_{+}(i),\tau(i)} - \sum_{i}e_{\sigma_{-}(i),\tau(i)}
\end{equation*}
lies in the kernel of~$\Bcal^\phi\TFP[\tilde\Acal](\Bcal')^{\phi'}$.  Call $\tilde m$ a \emph{lift} of~$m$.  This name is justified by the fact that
the set $\Lifts(m)$ of all such lifts is a $(\Fcal(\Bcal^{\phi}),\psi,\succeq_{\times})$-lift of~$m$.  Denote by
$\Lifts(\Mcal):=\bigcup_{m\in\Mcal}\Lifts(m)$ the set of all such lifts of all~$m\in\Mcal \subseteq \Zker\Bcal^\phi$.  We can similarly define the set $\Lifts(\Mcal')$, where $\Mcal' 
\subseteq \Zker(\Bcal')^{\phi'}$.

A second set of moves that we need is 
\begin{equation*}
  \Quads := \big\{ f_{i_{1},i_{2};j_{1},j_{2}} : \phi(i_{1})=\phi(i_{2})=\phi'(j_{1})=\phi'(j_{2})\big\},
\end{equation*}
where
$
  f_{i_{1},i_{2};j_{1},j_{2}} = 
  e_{i_1, j_1} + 
  e_{i_2, j_2} - 
  e_{i_1, j_2} - 
  e_{i_2, j_1}
$ 
and $e_{i,j}$ is the standard unit vector in $\zz^{\Bcal^{\phi}\TFP[\tilde\Acal](\Bcal')^{\phi'}}$
corresponding to $(b_i, b_j')$.

\begin{theorem}\cite{Sullivant07:TFPs}\label{thm:codimzero}
Suppose that $\Mcal$ and $\Mcal'$ are Markov bases for $\ker_\zz  \Bcal^{\phi}$ and 
$\ker_\zz  (\Bcal')^{\phi'}$, respectively. Then
\begin{equation}\label{eq:codimzeroGB}
\Lifts(\Mcal) \cup \Lifts(\Mcal') \cup  \Quads
\end{equation}
is a Markov basis for $\ker_\zz  ( \Bcal^{\phi}\TFP[\tilde\Acal](\Bcal')^{\phi'})$.  If,
in addition, $\Mcal$ and $\Mcal'$ are Gr\"obner bases for compatible preorders,
then, for any compatible additive preorder $\succeq_\times$ on $\zz^{\Bcal^{\phi}\TFP[\tilde\Acal](\Bcal')^{\phi'}}$, (\ref{eq:codimzeroGB})
is a Gr\"obner basis of $\ker_\zz (\Bcal^{\phi}\TFP[\tilde\Acal](\Bcal')^{\phi'})$.
\end{theorem}


\subsection{Projected fiber intersections}
\label{sec:PFI}

Next, we want to understand the geometry of the
projected fibers $\xi( \Fbf( \Bcal\TFP\Bcal', (b,b')))$.
These have a simple relation to the projected fibers
$ \phi( \Fbf(\Bcal, b))$ and $ \phi'( \Fbf(\Bcal', b'))$.

\begin{lemma}
  \label{lem:proj-vertices}
  $\xi( \Fbf( \Bcal\TFP\Bcal', (b,b'))) =  \phi( \Fbf(\Bcal, b)) \cap \phi'( \Fbf(\Bcal', b'))$.
\end{lemma}

\begin{proof}
  The first inclusion $\xi( \Fbf( \Bcal\TFP\Bcal', (b,b'))) \subseteq \phi( \Fbf(\Bcal, b)) \cap \phi'( \Fbf(\Bcal', b'))$ is
  trivial since $\psi( \Fbf( \Bcal\TFP\Bcal', (b,b'))) \subseteq \Fbf(\Bcal, b)$ and $\psi'( \Fbf( \Bcal\TFP\Bcal',
  (b,b'))) \subseteq \Fbf(\Bcal', b')$.
 
 If $ \phi( \Fbf(\Bcal, b)) \cap \phi'( \Fbf(\Bcal', b'))$ is non-empty, then let 
 $u\in\phi( \Fbf(\Bcal, b)) \cap \phi'( \Fbf(\Bcal', b'))$.  There exist
   $v\in\Fbf(\Bcal, b),v'\in\Fbf(\Bcal', b')$ with $u=\phi(v)=\phi'(v')$.  
   There is a unique representation
   $v=\sum_{i=1}^{r}e_{\sigma(i)}$ and $v'=\sum_{i=1}^{r'}e_{\sigma'(i)}$, 
   where $\sigma(i)\le\sigma(i+1)$ and
   $\sigma'(i)\le\sigma'(i+1)$.  Without loss of generality we may assume that $\phi$ and $\phi'$ are monotonically
   increasing functions on indices.  Then $\phi(\sigma(i))\le\phi(\sigma(i+1))$ 
   and $\phi'(\sigma'(i))\le\phi'(\sigma'(i+1))$.  The
   condition $\phi(v)=\phi'(v')$ implies $r=r'$ and 
   $\phi(\sigma(i))=\phi'(\sigma'(i))$ for all~$i$.  Let
   $w=\sum_{i=1}^{r}e_{\sigma(i),\sigma'(i)}$.  Then $\psi(w)=v$ and $\psi'(w)=v'$.  
   Therefore, $u \in \xi( \Fbf( \Bcal\TFP\Bcal', (b,b')))$.
\end{proof}

By Lemma~\ref{lem:proj-vertices}, the projected fibers $\xi( \Fbf( \Bcal\TFP\Bcal', (b,b')))$ are themselves intersections
of projected fibers of $\phi$ and~$\phi'$.
This motivates the following definition:
\begin{definition}
  A \emph{projected fiber intersection (PFI) Gr\"obner basis} of the toric fiber product is a projected fiber Gr\"obner
  basis for~$\xi$.
\end{definition}

A PFI Gr\"obner basis can be computed as an inequality Markov basis if the projected fibers
$\xi(\Fbf(\Bcal\TFP\Bcal',(b,b')))$ can be described in terms of linear equations and inequalities.  It easily follows
from Lemma~\ref{lem:proj-vertices} that this is the case if the same condition holds true for the projected fibers
$\phi( \Fbf(\Bcal, b))$ and $\phi'(\Fbf(\Bcal',b'))$.  Such inequality representations are easiest to obtain if
$\Nb\Bcal^{\phi}$ and $\Nb(\Bcal')^{\phi'}$ are both normal.  In fact, if both $\Nb\Bcal^{\phi}$ and
$\Nb(\Bcal')^{\phi'}$ are normal, then $\Nb(\Bcal\TFP\Bcal')^{\xi}$ is also normal.  This follows from
Lemma~\ref{lem:cod-zero-connects-fibers} and the fact that
normality is preserved in codimension-zero TFPs~\cite[Theorem~2.5]{EngstromKahleSullivant13:TFP-II} (but not in higher
codimension~\cite{KahleRauh2014:TFPs_and_SPs}).

\begin{remark}
  \label{rem:intersection-in-TF-powers}
  Suppose that $\Bcal=\Bcal'$ and~$\phi=\phi'$, and suppose that all projected fibers have an inequality description
  $\phi(\Fbf(\Bcal,b)) = \{u\in\Lcal + u_{0}(b) : D u \ge c(b) \}$, where the integer matrix $D$ and the lattice $\Lcal$ are
  independent of~$b$.  Then, an inequality Gr\"obner basis for $D$ is a PF Markov basis for~$\phi$ as well as a PFI
  Markov basis for the toric fiber product, because
  \begin{equation*}
    \phi(\Fbf(\Bcal,b))\cap\phi(\Fbf(\Bcal,b')) = \big\{u\in(\Lcal + u_{0}(b))\cap(\Lcal + u_{0}(b')) : D u \ge \max\{c(b),c(b')\} \big\}.
  \end{equation*}
\end{remark}

\subsection[Gluing lifts]{Gluing $\xi$-lifts from $\phi$-lifts and $\phi'$-lifts}
\label{sec:lifting-TFP}

Finally, we show how to lift moves $g\in\Zb^{t}$ along $\xi$ by gluing 
$\phi$-lifts and $\phi'$-lifts of~$g$.
Let $m\in\Zker\Bcal$ and $m'\in\Zker\Bcal'$ such that $\phi(m)=\phi'(m')=g$.  The goal of gluing is to construct a move
$\tilde m$ with $\psi(\tilde m)=m$ and~$\psi'(\tilde m)=m'$.  In general, $\tilde m$ will be larger than both $m$
and~$m'$ (in the sense of the $\ell_{1}$-norm or the degree, as defined later),
but the idea is to construct $\tilde m$ as small as possible.  The first step is to extend $m$ and $m'$ to make them compatible for gluing.

%
Let $v=\phi'(m^{\prime+})-\phi(m^{+})=\phi'(m^{\prime-})-\phi(m^{-})$.  Then
$\phi(m^{+}) + v^{+} = \phi'(m^{\prime+}) + v^{-}$ and $\phi(m^{-}) + v^{+} = \phi'(m^{\prime-}) + v^{-}$.
Choose vectors $\ol m^{+},\ol m^{-}\in\Nb^{n}$ and $\ol m^{\prime+},\ol m^{\prime-}\in\Nb^{n'}$ that satisfy $\phi(\ol
m^{+}-m^{+})=\phi(\ol m^{-}-m^{-})=v^{+}$ and $\phi(\ol m^{\prime+}-m^{\prime+})=\phi(\ol m^{\prime-}-m^{\prime-})=v^{-}$.
Since $\phi(\ol m^{+})=\phi'(\ol m^{\prime+})$ and $\phi(\ol m^{-})=\phi'(\ol m^{\prime-})$, there are functions
$\sigma,\sigma',\tau,\tau'$ satisfying
\begin{align*}
  \ol m^{+}&=\textstyle\sum_{i}e_{\sigma(i)},&   \ol m^{-}&=\textstyle\sum_{j}e_{\tau(j)}, & 
  \ol m^{\prime+}&=\textstyle\sum_{i}e_{\sigma'(i)},&   \ol m^{\prime-}&=\textstyle\sum_{j}e_{\tau'(j)}
\end{align*}
and $\phi(\sigma(i))=\phi'(\sigma'(i))$ and $\phi(\tau(j))=\phi'(\tau'(j))$.  Then the vector
\begin{equation*}
  \tilde m = \textstyle\sum_{i}e_{\sigma_{i},\sigma'_{i}} - \sum_{j}e_{\tau_{j},\tau'_{j}}
\end{equation*}
belongs to $\Zker(\Bcal\TFP\Bcal')$.  We call $\tilde m$ a \emph{glue} of $m$ and~$m'$.  Observe that indeed
$\psi(\tilde m)=m$ and $\psi'(\tilde m)=m'$.  See~\cite{EngstromKahleSullivant13:TFP-II} for a more detailed description
of the gluing procedure.

The set $\Glues(m,m')$ of all
glues of $m$ and $m'$ is finite, since $\phi^{-1}(v^{\pm})\cap\Nb^{n}$ and $\phi^{\prime-1}(v^{\pm})\cap\Nb^{n'}$ are finite.  For
any $\Mcal\subseteq\Zker\Bcal$, $\Mcal'\subseteq\Zker\Bcal'$ denote by $\Glues(\Mcal,\Mcal')$ the set of all glues of
compatible elements of $\Mcal$ and~$\Mcal'$.
The gluing construction has the following crucial property:
\begin{lemma}
  \label{lem:glue-property}
  Let $m\in\Zker\Bcal$, $m'\in\Zker\Bcal'$ with $\phi(m)=\phi'(m')$, and let $w\in\Nb^{\Bcal\TFP\Bcal'}$.  If
  $\psi(w)+m\ge 0$ and $\psi'(w)+m'\ge0$, then there exists $\tilde m\in\Glues(m,m')$ with $w + \tilde m\ge 0$.
\end{lemma}

\begin{proof}
This is a restatement of Lemma 4.8 of \cite{EngstromKahleSullivant13:TFP-II}. 
\end{proof}

\begin{lemma}
  \label{lem:gluing-lifts}
  Let $\Mcal\subset\Zker\Bcal$ and $\Mcal'\subset\Zker\Bcal'$ be $(\Fcal(\Bcal),\phi,\succeq_{\Bcal})$- and
  $(\Fcal(\Bcal'),\phi',\succeq_{\Bcal'})$-lifts of a $\succeq_{\Acal}$-Gr\"obner basis $\Gcal$ of
 $\xi( \Fcal( \Bcal\TFP\Bcal' ) )$.  Then
  $\Glues(\Mcal,\Mcal')$ is a $(\Fcal(\Bcal\TFP\Bcal'),\xi,\succeq_{\times})$-lift of $\Gcal$.
\end{lemma}

\begin{proof}
  Suppose that $w_{1},w_{2}\in\Nb^{\Bcal\TFP\Bcal'}$ satisfy $\xi(w_{1}-w_{2})=g\in\Gcal$ and
  $(\Bcal\TFP\Bcal')(w_{1}-w_{2})=0$.  Then $v_{1}=\psi(w_{1})$ and $v_{2}=\psi(w_{2})$ satisfy $\phi(v_{1}-v_{2})=g$
  and $\Bcal(v_{1}-v_{2})=0$.  Since $\Mcal$ lifts~$\Gcal$, there are $m\in\Mcal$ and $m_{0}\in\ker\phi$ as in
  Definition~\ref{def:lift}.  Similarly, $v'_{1}=\psi'(w_{1})$ and $v'_{2}=\psi'(w_{2})$ satisfy
  $\phi'(v'_{1}-v'_{2})=g$ and $\Bcal'(v'_{1}-v'_{2})=0$, so we can find $m'\in\Mcal'$ and $m'_{0}\in\ker\phi'$
  as in Definition~\ref{def:lift}.  By Lemma~\ref{lem:glue-property}, there are $\tilde m_{0}\in\Glues(m_{0},m_{0}')$
  and $\tilde m\in\Glues(m,m')$ 
  such that $w_{1}+\tilde m_{0}\ge 0$ and $w_{1}+\tilde m_{0}+\tilde m\ge 0$.  Let $\tilde m_{1}:= w_{2} - w_{1}-\tilde
  m_{0}-\tilde m$.  Then $\xi(\tilde m_{1})=\phi(v_{2}-v_{1}-m_{0}-m)=0$.
%
  Thus, $\xi(\tilde m_{0})=\xi(\tilde m_{1})=0$, $\xi(\tilde m)=g$ and $(\Bcal\TFP\Bcal')\tilde m_{0}=(\Bcal\TFP\Bcal')\tilde
  m=(\Bcal\TFP\Bcal')\tilde m_{1}=0$.  Moreover the sequence $w_{1}$, $w_{1} + \tilde m_{0}$, $w_{1} + \tilde m_{0} +
  \tilde m$ is non-increasing, due to our compatibility requirements.
  Hence the conditions of Definition~\ref{def:lift} are verified.
\end{proof}

The results of this section are related to the following notion from~\cite{EngstromKahleSullivant13:TFP-II}.
\begin{definition}
  Two Markov bases $\Mcal\subset\Zker\Bcal$, $\Mcal'\subset\Zker\Bcal'$ satisfy the \emph{compatible projection
    property}, if the graph $\phi(\Fbf(\Bcal,b)_{\Mcal})\cap\phi'(\Fbf(\Bcal',b')_{\Mcal'})$ is
  connected for all~$b\in\Nb\Bcal$, $b'\in\Nb\Bcal'$.
  Here, $\phi(\Fbf(\Bcal,b)_{\Mcal})$ denotes the image of the graph $\Fbf(\Bcal,b)_{\Mcal}$
  under~$\phi$, as defined in the introduction.
\end{definition}
Theorem~4.9 in~\cite{EngstromKahleSullivant13:TFP-II} says that if $\Mcal$ and $\Mcal'$ have the compatible projection
property, then the union of $\Glues(\Mcal,\Mcal')$ and a Markov basis of the associated codimension-zero toric fiber
product is a Markov basis of $\Bcal\TFP\Bcal'$.
The proof of Lemma~\ref{lem:gluing-lifts} basically shows that, if $\Mcal$ and $\Mcal'$ lift a PFI Markov basis~$\Gcal$,
then $\phi(\Fbf(\Bcal,b)_{\Mcal})\cap\phi(\Fbf(\Bcal',b')_{\Mcal'}) =
\xi(\Fbf(\Bcal\TFP\Bcal',(b,b')))_{\Gcal}$.  Hence, in this case, $\Mcal$ and $\Mcal'$ have the compatible projection
property.

The compatible projection property is weaker than the property of being lifts.  Sometimes it is possible to find subsets
of lifts which still satisfy the compatible projection property.  In this way, a smaller Markov basis
of~$\Bcal\TFP\Bcal'$ can be found.  For an example see Section~\ref{sec:K4-e}.

We conclude this section with another result from~\cite{EngstromKahleSullivant13:TFP-II}:
\begin{lemma}[Theorem~4.2 in~\cite{EngstromKahleSullivant13:TFP-II}]
  \label{lem:slow-varying-TFP}
  Let $\Bcal\TFP\Bcal'$ be a codimension-one toric fiber product, and let $\Mcal,\Mcal'$ be slow-varying Markov bases of
  $\Bcal$ and $\Bcal'$.  Then $\Mcal$ and $\Mcal'$ satisfy the compatible projection property.
\end{lemma}


\subsection{A simple example}
\label{sec:simple-ex}

  Consider the matrix $\Bcal$ and the map $\phi$ from Example~\ref{ex:simple-lifting}.  Then $\phi$ corresponds to the map
  \begin{equation*}
    1\mapsto 1, \quad 2\mapsto 2, \quad 3\mapsto 2, \quad 4\mapsto 3.
  \end{equation*}
  Let $\Bcal'=\Bcal$, and let
  \begin{equation*}
    \phi' =
    \scalebox{0.9}{$
    \begin{pmatrix}
      0 & 1 & 1 & 0 \\
      1 & 0 & 0 & 0 \\
      0 & 0 & 0 & 1
    \end{pmatrix}$}
  \end{equation*}
  be the map that arises from $\phi$ by switching the role of the first two coordinates in the image.  The corresponding
  toric fiber product is
  \begin{equation*}
    \Bcal\times_{\Acal}\Bcal' =
    \scalebox{0.9}{$
    \begin{pmatrix}
      1 & 1 & 1 & 1 & 1 \\
      0 & 0 & 0 & 1 & 2 \\
      1 & 1 & 1 & 1 & 1 \\
      0 & 1 & 0 & 0 & 2
    \end{pmatrix}$}.
  \end{equation*}

  Using the symmetry between $\phi$ and~$\phi'$, the projected fiber intersections can be described as the set of integer solutions
  of inequalities of the form
  \begin{gather*}
    y_{1} \ge 0,
    \qquad
    y_{2} \ge 0,
    \\
    y_{1} + y_{2} \le c_{1},
    \qquad
    y_{1} + y_{2} \ge c_{2},
    \qquad
    2 y_{1} + y_{2} \le c_{3},
    \qquad
    y_{1} + 2 y_{2} \le c_{4},
  \end{gather*}
  corresponding to the matrix
  \begin{equation*}
    D = \scalebox{0.9}{$
    \begin{pmatrix}
      1 &  0 \\
      0 &  1 \\
      -1 & -1 \\
      1 &  1 \\
      -2 & -1 \\
      -1 & -2
    \end{pmatrix}$}.
  \end{equation*}
  The Markov basis of the lattice generated by the columns of $D$ contains three elements:
  \begin{equation*}
    (0, 1 ,-1, 1 ,-1 ,-2), \quad (1 ,-1, 0, 0 ,-1, 1), \quad (1, 0 ,-1, 1 ,-2 ,-1).
  \end{equation*}
  The inverse images under $D$ are
  $ 
    (0, 1), 
    (1, -1)$ and $
    (1,0),
  $ 
  and so the PF Markov basis is given by
  \begin{equation*}
    \Gcal = \Big\{ g_{1}=(0,1,-1), \quad g_{2}=(1,-1,0), \quad g_{3}=(1,0,-1) \Big\}.
  \end{equation*}

  Each move in $\Gcal$ has a single $\phi$-lift, and the lifted Markov basis is
  \begin{equation*}
    \Mcal = \Big\{ m_{1}=(0,-1,2,-1), \quad m_{2}=(1,-1,0,0), \quad m_{3}=(1,-2,2,-1) \Big\}.
  \end{equation*}
  Both $\Gcal$ and $\Mcal$ are symmetric under the exchange of $y_{1}$ and~$y_{2}$, and so $\Mcal$ is also a
  $\phi'$-lift of~$\Gcal'$.
  We have
  \begin{equation*}
    \phi(m_{1}) = g_{1} = \phi'(m_{3}), \qquad
    \phi(m_{2}) = g_{2} = \phi'(-m_{2}), \qquad
    \phi(m_{3}) = g_{3} = \phi'(m_{1}).
  \end{equation*}
  In each case, one can check that there is just a single glued element:
  \begin{align*}
    \text{Glues}(m_{1},m_{3}) &=
    \Big\{
    \hat m_{1}=(-2, 2, -1, 2, -1)
    \Big\},
    \\
    \text{Glues}(m_{2},-m_{2}) &=
    \Big\{
    \hat m_{2}=(1, 0, -1, 0, 0)
    \Big\},
    \\
    \text{Glues}(m_{3},m_{1}) &=
    \Big\{
    \hat m_{3}=(-1, 2, -2, 2, -1)
    \Big\}.
  \end{align*}
  Thus these three moves form a Markov basis of the TFP.  In fact, it suffices to take the first two moves:
  Suppose that we want to apply $\hat m_{3}$: Then $x_{1},x_{5}\ge 1$ and $x_{3}\ge 2$.  Hence we can apply $\hat m_{2}$.
  The result has $x_{1}\ge 2$ and $x_{3}\ge 2$.  Hence we can apply~$\hat m_{1}$.  But $\hat m_{3}=\hat m_{1} + \hat
  m_{2}$.

  In fact, \verb|4ti2| gives the Markov basis $\{\hat m_{2}, \hat m_{4}\}$ with
    $\hat m_{4} = ( 3, -2, 0, -2, 1)$.
  Observe that $\hat m_{4} = - \hat m_{1} + \hat m_{2}$, and an argument as above shows that $\{\hat m_{2}, \hat m_{4}\}$
  is equivalent to $\{\hat m_{1},\hat m_{2}\}$.


\section{Application to Hierarchical Models}
\label{sec:models}

This section explores our main applications to constructing Markov bases of hierarchical models and gives two more
complex examples.  The Markov basis of the four-cycle is studied in Section~\ref{sec:4-cycle}.  We collect known results
about the no three-way interaction model that allow to compute kernel Gr\"obner bases and PFI Gr\"obner bases for some
values of the parameter~$d$.  Section~\ref{sec:K4-e} contains an example of a codimension-one toric fiber product where
the associated semigroup has holes.  Thus, a PFI Markov basis cannot directly be computed as an inequality Markov basis,
but it is not difficult to adjust our ideas.  The example also illustrates that the Markov basis obtained through our
algorithm is in general too large.  A detailed analysis shows that not all lifted moves are necessary.

Before delving into the examples, let us fix the notation.
Let $\Gamma$ be a simplicial complex with vertex set $V$, and let $d \in \zz^{V}_{\geq 2}$.
For $F \subseteq V$ let $D_{F} :=  \prod_{j \in F} [d_{j}]$.
For each $i = (i_{j})_{j \in V} \in D_{V}$ let
$i_{F}  = (i_{j})_{j \in F}$, be the subvector with
index set $F$.
Let $\calb_{\Gamma,d}$ be the matrix that consists of the following $\# D_{V}$ columns,
one for each $i \in D_{V}$:
$$
b_{i} := \bigoplus_{F \in \text{facet}(\Gamma)}  e_{i_{F}}  \in \;
\bigoplus_{F \in \text{facet}(\Gamma)}  \zz^{D_{F}},
$$ 
where $e_{i_{F}}$ denotes the $i_{F}$th standard unit vector in $\zz^{D_{F}}$.
The matrix $\Bcal_{\Gamma,d}$ is the \emph{design matrix of the hierarchical model specified by~$\Gamma$ and~$d$}.  If $d_{i}=2$ for all~$i\in V$, we speak of a \emph{binary} hierarchical model.
The fibers of $\Bcal_{\Gamma,d}$ have the following natural interpretation: For each facet $F\in\Gamma$, the linear map
of $\Bcal_{\Gamma,d}$ computes the $F$-margins, and thus, the fibers of $\Bcal_{\Gamma,d}$ correspond to sets of
non-negative tensors where a certain number of margins (determined by~$\Gamma$) is fixed.
%

The kernel of $\calb_{\Gamma,d}$ lies in~$\zz^{D_{V}}$.  Elements of $\Zb^{D_{V}}$ are often written in tableau notation,
as the difference of two matrices of indices.  For example, the vector
$$
2e_{111} + e_{222} - e_{112} - e_{121} - e_{211}
$$
is represented in tableau notation as
$$
\scalebox{0.9}{$
\begin{bmatrix}
1  1  1 \\
1  1  1 \\
2  2  2 \\
\end{bmatrix}$} -
\scalebox{0.9}{$
\begin{bmatrix}
1  1  2 \\
1  2  1 \\
2  1  1 \\
\end{bmatrix}$}.
$$

The matrix $\calb_{\Gamma,d}$ is a toric fiber product whenever $\Gamma$ is missing edges, and the associated
codimension zero toric fiber product can also be described by a simplicial complex obtained by filling in the separator:

\begin{proposition}[Propositions~5.1 and 5.2 in~\cite{EngstromKahleSullivant13:TFP-II}]
\label{prop:tfp-hierarchical}
Let $\Gamma$ be a simplicial complex on $V$.  Let $V_{1}, V_{2} \subseteq V$ such
that $V = V_{1} \cup V_{2}$, and $\Gamma = \Gamma|_{V_{1}} \cup \Gamma|_{V_{2}}$.
Let $S = V_{1} \cap V_{2}$.  Then:
\begin{enumerate}
\item $\calb_{\Gamma,d} = \calb_{\Gamma|_{V_{1}}, d_{V_{1}}} \TFP[\calb_{\Gamma|_{S}, d_{S}}] \calb_{\Gamma|_{V_{2}},
    d_{V_{2}}}$.  The $\calb_{\Gamma|_{S}, d_{S}}$-grading of $\calb_{\Gamma|_{V_{1}}, d_{V_{1}}}$ and
  $\calb_{\Gamma|_{V_{2}}, d_{V_{2}}}$ is given by the marginalization maps $\phi:\Zb^{D_{V_{1}}}\to\Zb^{D_{S}}$ and
  $\phi':\Zb^{D_{V_{2}}}\to\Zb^{D_{S}}$.  Thus, the $\calb_{\Gamma|_{S}, d_{S}}$-grading of the product
  $\calb_{\Gamma,d}$ is also given by the marginalization map $\xi:\Zb^{D_{V}}\to\Zb^{D_{S}}$.
\item Let $\tilde\Gamma := \Gamma \cup 2^S$.  Then
$
\calb^{\phi}_{\Gamma|_{V_{1}}, d_{V_{1}}}  \times_{\tilde\calb_{2^{S}, d_{S}}}
\calb^{\phi'}_{\Gamma|_{V_{2}}, d_{V_{2}}}  =  \Bcal_{\tilde\Gamma, d}. 
$
\end{enumerate}
\end{proposition}


Normality of $\nn\Bcal_{\tilde\Gamma, d}$ plays an important role in easily
determining a PF Markov basis.
Only in certain special cases do we possess classifications of normal hierarchical models.

\begin{theorem}\textup{\cite{BHIKS11:Challenging_Hilbert_bases}}
  \label{thm:C3-normal}
Let $\Gamma = [12][13][23]$ be a $3$-cycle (also called ``no three-way interaction model").
Then $\nn\Bcal_{\Gamma, d}$ is normal if and only if, up to symmetry, $d$ is one of:
  \begin{equation*}
    (3,4,4),\quad
    (3,4,5),\quad
    (3,5,5),\quad
    (2,p,q),\quad
    (3,3,q),\quad
    \text{ with $p, q \in \nn$}.
  \end{equation*}
\end{theorem}


\subsection{The 4-cycle}
\label{sec:4-cycle}

\begin{figure}
  \centering
  \begin{tikzpicture}[scale=0.5]
    \node at (-0.3,1) {a)};
    \fill (1,1) circle (4pt);
    \fill (0,0) circle (2pt);
    \fill (1,-1) circle (4pt);
    \draw (1,1) -- (0,0) -- (1,-1);
    \node at (2,0) {$\times$};
    \fill (3,1) circle (4pt);
    \fill (4,0) circle (2pt);
    \fill (3,-1) circle (4pt);
    \draw (3,1) -- (4,0) -- (3,-1);
    \node at (5,0) {$=$};
    \fill (7,1) circle (4pt);
    \fill (6,0) circle (2pt);
    \fill (7,-1) circle (4pt);
    \fill (8,0) circle (2pt);
    \draw (7,1) -- (6,0) -- (7,-1) -- (8,0) -- cycle;
  \end{tikzpicture}
  \hfil
  \begin{tikzpicture}[scale=0.5]
    \node at (-0.3,1) {b)};
    \fill (1,1) circle (4pt);
    \fill (0,0) circle (2pt);
    \fill (1,-1) circle (4pt);
    \draw (1,1) -- (0,0) -- (1,-1) -- (1,1);
    \node at (2,0) {$\times$};
    \fill (3,1) circle (4pt);
    \fill (4,0) circle (2pt);
    \fill (3,-1) circle (4pt);
    \draw (3,1) -- (4,0) -- (3,-1) -- (3,1);
    \node at (5,0) {$=$};
    \fill (7,1) circle (4pt);
    \fill (6,0) circle (2pt);
    \fill (7,-1) circle (4pt);
    \fill (8,0) circle (2pt);
    \draw (7,1) -- (6,0) -- (7,-1) -- (8,0) -- (7,1) -- (7,-1);
  \end{tikzpicture}
  \caption{a) The 4-cycle $C_{4}$ as a toric fiber product.  b) $\tilde C_{4}$ as a codimension-zero toric
    fiber product.}
  \label{fig:4cycle}
\end{figure}
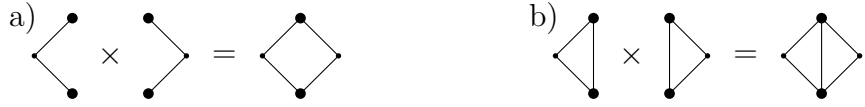
In this section, we use the toric fiber product and lifting techniques
to construct Markov and Gr\"obner bases of the $4$-cycle model $\Gamma = C_{4} := [12][13][24][34]$,
for various values of~$d$.
In the case $d_{1}=d_{4}=2$ (two opposite binary vertices), Markov bases were already computed in~\cite{KRS14:Positive_Margins_and_Prim_Dec}
and~\cite{EngstromKahleSullivant13:TFP-II}.
We apply Proposition~\ref{prop:tfp-hierarchical} with $V_1 = \{1,2,3\}$ and
$V_2 = \{2,3,4\}$, so that $\Gamma_1 = [12][13]$ and
$\Gamma_2 = [24][34]$; see Figure~\ref{fig:4cycle}a).
In fact, our results easily generalize to the complete bipartite graph~$K_{2,N}$, which arises by iterating the toric fiber product, as detailed in Section~\ref{sec:degree-bounds}.

\subsubsection{The associated codimension zero product}

By Proposition~\ref{prop:tfp-hierarchical}, the associated codimension zero product is the hierarchical model on the simplicial complex $\tilde C_{4} := [12][13][23][24][34]$; see Figure~\ref{fig:4cycle}b).  This is obtained by gluing the two
triangles $\tilde\Gamma_{1}$ and $\tilde\Gamma_{2}$ along an edge.  Theorem~\ref{thm:codimzero} can be used to construct
the Markov basis in this case, provided we know the Markov bases for $\tilde\Gamma_{1}$ and~$\tilde\Gamma_{2}$.
The Markov bases of triangles are not known in general, but
are simple to compute in some instances~\cite{DiaconisSturmfels1998}.

\begin{theorem}
\label{thm:mb2xpxq}
Let $C_{3}:=[12][13][23]$ be a triangle, and let $d = (p,2,r)$.  For any sequences $i : = i_1, \ldots, i_k \in [p]$ and $j:= j_1, \ldots, j_k \in [r]$ with pairwise
distinct entries, let
$$
f_{i,j} :=
\sum_{t = 1}^k (
e_{i_t,1, j_t} - e_{i_t,2, j_t} + e_{i_t,2, j_{t+1}} - e_{i_t,1, j_{t+1}} )
$$
where $j_{k+1} := j_1$.
Then
$$\Mcal =  \{f_{i,j}  :  k = 2, \ldots, \min(p,r), i : = i_1, \ldots, i_k \in [p], j:= j_1, \ldots, j_k \in [r] \}$$
is the Graver basis of $\ker_\zz \Bcal_{C_{3}, d}$.  In particular, $\Mcal$ is a Gr\"obner basis for any additive preorder.
\end{theorem}

With the help of Theorems~\ref{thm:mb2xpxq} and~\ref{thm:codimzero}, it is easy to
construct a Gr\"obner basis of $\tilde C_{4}$ when $d = (p,2,r,q)$.

\begin{example}\label{ex:codim0p23q}
Let $d = (p,2,3,q)$.  A Gr\"obner basis for $\ker_\zz \Bcal_{\tilde C_{4},d}$
consists of the moves
$$
\scalebox{0.9}{$
\begin{bmatrix}
a_{1}  b   c  e_{1} \\
a_{2}  b   c  e_{2} 
\end{bmatrix}$}-
\scalebox{0.9}{$
\begin{bmatrix}
a_{1} b c e_{2} \\
a_{2} b c e_{1} 
\end{bmatrix}$},
\scalebox{0.9}{$
\begin{bmatrix}
a_{1} 1  c_1 e_{1} \\
a_{2} 2  c_1 e_{2} \\
a_{2} 1  c_2 e_{3} \\
a_{1} 2  c_2 e_{4} \\
\end{bmatrix}$}-
\scalebox{0.9}{$
\begin{bmatrix}
a_{2} 1  c_1 e_{1} \\
a_{1} 2  c_1 e_{2} \\
a_{1} 1  c_2 e_{3} \\
a_{2} 2  c_2 e_{4} \\
\end{bmatrix}$},
\scalebox{0.9}{$
\begin{bmatrix}
a_{1} 1  c_1 e_{1} \\
a_{2} 2  c_1 e_{2} \\
a_{3} 1  c_2 e_{2} \\
a_{4} 2  c_2 e_{1} \\
\end{bmatrix}$}-
\scalebox{0.9}{$
\begin{bmatrix}
a_{1} 1  c_1 e_{2} \\
a_{2} 2  c_1 e_{1} \\
a_{3} 1  c_2 e_{1} \\
a_{4} 2  c_2 e_{2} \\
\end{bmatrix}$},
$$
$$
\scalebox{0.9}{$
\begin{bmatrix}
a_{1} 1  1 e_{1} \\
a_{2} 1  2 e_{2} \\
a_{3} 1  3 e_{3} \\
a_{2} 2  1 e_{4} \\
a_{3} 2  2 e_{5} \\
a_{1} 2  3 e_{6} \\
\end{bmatrix}$} -
\scalebox{0.9}{$
\begin{bmatrix}
a_{2} 1  1 e_{1} \\
a_{3} 1  2 e_{2} \\
a_{1} 1  3 e_{3} \\
a_{1} 2  1 e_{4} \\
a_{2} 2  2 e_{5} \\
a_{3} 2  3 e_{6} \\
\end{bmatrix}$},
\scalebox{0.9}{$
\begin{bmatrix}
a_{1} 1  1 e_{1} \\
a_{2} 1  2 e_{2} \\
a_{3} 1  3 e_{3} \\
a_{4} 2  1 e_{2} \\
a_{5} 2  2 e_{3} \\
a_{6} 2  3 e_{1} \\
\end{bmatrix}$} -
\scalebox{0.9}{$
\begin{bmatrix}
a_{1} 1  1 e_{2} \\
a_{2} 1  2 e_{3} \\
a_{3} 1  3 e_{1} \\
a_{4} 2  1 e_{1} \\
a_{5} 2  2 e_{2} \\
a_{6} 2  3 e_{3} \\
\end{bmatrix}$},
$$
where $a, a_1,a_2, \ldots, a_6 \in [p]$, $b \in [2]$, $c, c_1, c_2 \in [3]$
and $e,e_1, e_2, \ldots, e_6 \in [q]$.
This Gr\"obner basis works for any choice of preorders
$\succeq_{\times},\succeq_{\Bcal},\succeq_{\Bcal'},\succeq_{\Acal}$ that satisfy the compatibility conditions from
Section~\ref{sec:TFP}.
\qed
\end{example}

\subsubsection{The projected fibers}

Next we describe a PF Gr\"obner basis, associated with the projection
$$\phi: \zz^{D_{V_1}}  \rightarrow  \zz^{D_{23}}, \quad  e_{i_1,i_2,i_3}  \rightarrow
e_{i_2, i_3},$$ to the missing $[23]$ margin.  We first need a suitable description of the projected fibers $\phi (\Fcal(
\Bcal_{\Gamma_1, d_{V_1}}))$.  We are mostly interested in the case where we can find an inequality description.  Then,
according to Remark~\ref{rem:intersection-in-TF-powers}, an inequality Gr\"obner basis for this inequality description
will be a PFI Gr\"obner basis.
We assume that $d = (p,2,r)$, so that $\nn\Bcal_{C_{3}, d_{V_1}}$ 
is normal (Theorem~\ref{thm:C3-normal}).  The facets of $\Rb_{\ge}\Bcal_{C_{3}, d_{V_1}}$ are well-studied. 
In the case $d = (p,2,r)$ the result is~\cite{Vlach86:3D_planar_transportation_problem}:

\begin{proposition}\label{prop:3cycleineq}
Let $d = (p,2,r)$.  The cone
 $\rr_\geq \Bcal_{C_{3}, d}$ is the solution to the following
 system of inequalities:
 \begin{gather*}
y^{12}_{ij} \geq 0, \quad y^{13}_{ik} \geq 0, \quad  y^{23}_{jk} \geq 0, \\
y^1_i  -  y^{12}_{ij}  \geq 0, \quad  y^2_j  -  y^{23}_{jk}  \geq 0, \quad y^3_k  -  y^{13}_{ik} \geq 0, \\
y^1_i  -  y^{13}_{ik}  \geq 0, \quad  y^2_j  -  y^{12}_{ij} \geq 0,  \quad y^3_k  -  y^{23}_{jk} \geq 0, \\
y^\emptyset - y^1_i - y^2_j  + y^{12}_{ij}  \geq 0, \quad 
y^\emptyset - y^1_i - y^3_k  + y^{13}_{ik}  \geq 0, \quad 
y^\emptyset - y^2_j - y^3_k  + y^{23}_{jk}  \geq 0 \\
\sum_{i \in A, k \in B} y^{13}_{ik} + \sum_{i \in A} ( y^{12}_{i2} - y^1_i) + \sum_{k \in B} (y^{23}_{2k} - y^3_k) -
y^2_2 + p^\emptyset \geq 0
\\
\sum_{i \in A, k \in B} y^{13}_{ik} - \sum_{i \in A} ( y^{12}_{i2} + y^1_i) - \sum_{k \in B} (y^{23}_{2k} + y^3_k) +
y^2_2 \geq 0
\\
-\sum_{i \in A, k \in B} y^{13}_{ik} + \sum_{i \in A} ( y^{12}_{i2} - y^1_i) -\sum_{k \in B} (y^{23}_{2k} - y^3_k) -
y^2_2 \geq 0
\\
-\sum_{i \in A, k \in B} y^{13}_{ik} - \sum_{i \in A} ( y^{12}_{i2} - y^1_i) + \sum_{k \in B} (y^{23}_{2k} - y^3_k) -
y^2_2 \geq 0.
\end{gather*}
Here, $y^{F}_{i_{F}}$ is the coordinate corresponding to the unit vector $e_{i_{F}}$ corresponding to the
$F$-marginal taking the value~$i_{F}$.
\end{proposition}

The projection $\phi$ onto the $[23]$ marginal amounts to setting all
of the variables in the inequalities that appear in the $[12][13]$ model to
fixed numbers and looking at the induced inequality system on the other variables.
In particular, the only indeterminates that do not appear in $[12][13]$ are
the indeterminates~$y^{23}_{jk}$. 
Using the relations $y^{23}_{2k}=y^{3}_{k}-y^{23}_{1k}$ and $y^{23}_{2r} = y^{2}_{2} - \sum_{k=1}^{r-1}y^{23}_{2k}$ we
can eliminate all $y^{23}_{2k}$ and $y^{23}_{2r}$ and restrict attention to the 
indeterminates $y^{23}_{1k}$ with $k \in \{1, \ldots, r-1\}$.  The
linear forms constraining these coordinates are all of the form
$ 
\sum_{k \in B} y^{23}_{1k}
$ 
for some $B \subseteq \{1, \ldots, r-1\}$.  Hence we have to solve the following problem:

\begin{problem}
  \label{prob:sumsineqMB}
  Fix an integer $t$.  For each $u,l \in \zz^{2^{[t]}}$ let
  \begin{equation*}
    \Sbf(u,l) :=   \Big\{   x \in \zz^{t} :     l_{A}  \leq  \sum_{i \in A} x_{i}  \leq  u_{A}  \mbox{ for all }
    A \subseteq  [t] \Big\}.
  \end{equation*}
  We wish to find inequality Markov bases for the collection $\Fcal_{t} := \{\Sbf(u,l)  : u, l   \in \zz^{2^{[t]}}  \}$.
\end{problem}
By Lemma~\ref{lem:MB-is-GB}, a Markov basis for Problem~\ref{prob:sumsineqMB} is also a Gr\"obner basis with respect
to any additive preorder.
Note that the inequality system in Problem~\ref{prob:sumsineqMB} does not depend on~$p$.
Hence, if we solve Problem~\ref{prob:sumsineqMB} for some $t$, we have a PF Markov basis of $\Gamma_{1}=[12][13]$
for all triples $(p,2,t+1)$ for all~$p$.  The resulting polytopes whose integer points
we are trying to connect are called generalized permutahedra~\cite{Postnikov2009}.

For $t = 1$, the solution is trivial, a Markov basis consists of two moves $\{\pm 1\}$.
For $t = 2$, by Example~\ref{ex:simple-ilMB}, the Markov basis
consists of six moves $\big\{ \pm(1,0), \pm(0,1), \pm(1,-1)\big\}$.
Note, however, that for the purposes of lifting, we should really consider this
as part of the $2 \times(t+1)$ matrix, whose row and column sums are equal to zero.  Hence,
we must complete these vectors to $2 \times(t+1)$ matrices with this property.
The Markov basis for $t = 1$ becomes 
$\pm\begin{smallpmatrix}  1 & -1 \\  -1 & 1  \end{smallpmatrix} $.
For $t =2$, up to the natural $\zz_{2} \times S_{3}$ symmetry, the inequality Markov basis consists
of a single move $\begin{smallpmatrix} 1 & -1 & 0  \\ -1 & 1 & 0 \end{smallpmatrix}$.

\newcommand\hlinespacing{^{\rule{0pt}{1.4ex}}_{\rule[-0.5ex]{0pt}{0ex}}}  
\begin{table}
  \centering
  \begin{tabular}{c|l}
    \hline
    \rule{0pt}{2.4ex} $r = t+1$ &  \mbox{ new moves}  \\
    \hline
    2  &  \scalebox{0.9}{$\begin{pmatrix}  1 & -1 \\  -1 & 1  \end{pmatrix}\hlinespacing$}  \\
    \hline
    3  & $ \emptyset\hlinespacing$  \\  \hline
    4  & \scalebox{0.9}{$ \begin{pmatrix} 1 & 1 & -1 & -1 \\ -1 & -1 & 1 & 1 \end{pmatrix}\hlinespacing$} \\
    \hline
    5 &  \scalebox{0.9}{$ \begin{pmatrix} 2 & 1 & -1 & -1 & -1 \\ -2 & -1 & 1 & 1 & 1 \end{pmatrix}\hlinespacing$} \\
    \hline
    6 &  \scalebox{0.9}{$ \begin{pmatrix} 1 & 1 & 1 & -1 & -1 & -1 \\ -1 & -1 & -1 & 1 & 1 & 1 \end{pmatrix}$}  
         \scalebox{0.9}{$ \begin{pmatrix} 2 & 1 & 1 & -2 & -1 & -1 \\ -2 & -1 & -1 & 2 & 1 & 1 \end{pmatrix}\hlinespacing$} \\
      &  \scalebox{0.9}{$ \begin{pmatrix} 2 & 2 & -1 & -1 & -1 & -1 \\ -2 & -2 & 1 & 1 & 1 & 1 \end{pmatrix}$}
         \scalebox{0.9}{$ \begin{pmatrix} 3 & 1 & -1 & -1 & -1 & -1 \\ -3 & -1 & 1 & 1 & 1 & 1 \end{pmatrix}$}  \\
      &  \scalebox{0.9}{$ \begin{pmatrix} 3 & 1 & 1 & -2 & -2 & -1 \\ -3 & -1 & -1 & 2 & 2 & 1 \end{pmatrix}$}
         \scalebox{0.9}{$ \begin{pmatrix} 3 & 2 & -2 & -1 & -1 & -1 \\ -3 & -2 & 2 & 1 & 1 & 1 \end{pmatrix}$}
    \\
    \hline
  \end{tabular}

  \caption[Markov bases for~$\Fcal_{t}$ up to symmetry.]{The PF Markov bases for various values of~$r$ up to symmetry.
    Each Markov basis contains the moves from the previous rows padded with columns of zeros.
    For example, the Markov basis of $r=3$ consists of
    $\begin{smallpmatrix}1&-1&0\\-1&1&0\end{smallpmatrix}$, $\begin{smallpmatrix}1&0&-1\\-1&0&1\end{smallpmatrix}$ and $\begin{smallpmatrix}0&1&-1\\0&-1&1\end{smallpmatrix}$.
  }
\label{tab:St}
\end{table}

We computed the Markov bases for various values of $t$ using~\verb|4ti2|.  Table~\ref{tab:St} summarizes our
results, classifying the elements in the Markov basis up to symmetry.  In the table, the moves are already converted into the form of $2\times(t+1)$-tables in which we need them later as a PF Markov basis.
A Markov basis of~$\Fcal_{t}$ can be obtained by dropping the second row and the last column.

We do not know a general solution to Problem \ref{prob:sumsineqMB},
and we think it will be an interesting challenge to try to 
find a general form for the inequality Markov basis in this case.


\subsubsection{Lifting the IPF Gr\"obner basis}

Next, supposing that we have solved Problem~\ref{prob:sumsineqMB} for $t=r-1$, we explain how to lift 
along the map~$\phi$.
\begin{proposition}
  \label{prop:lifting-3string}
  Let $b$ be a $2 \times r$-matrix from a PF Gr\"obner basis
  of $\xi(\Fcal(\Bcal_{C_{4},(p,2,r,q)}))$. 
  There is a lifting of $b$ along $\phi$ to $(p,2,r)$-arrays in which the combinatorial types
  are in bijections with directed acyclic multigraphs with vertex set $[r]$ such that
  for each vertex $i \in [r]$,  $\operatorname{outdeg}(i) - \operatorname{indeg}(i) =  b'_{i}$.  
\end{proposition}

The bijection in the proposition is as follows:
associate to such a multigraph $G$,  
and a collection of elements $a_{ij} \in [p]$, one for each edge $i \to j \in E(G)$,
the move
$$
\sum_{i \to j \in E(G) } (e_{a_{ij},1,i} + e_{a_{ij},2,j}  -  e_{a_{ij},1,j} - e_{a_{ij},2,i}).
$$

\begin{proof}
If we remove the restriction that $G$ does not contain directed cycles,
the set of all such moves produced contains all vectors in $\ker_{\zz} \Bcal_{\Gamma_{1}, (p,2,r)}$ 
that project to~$b$.  If a graph $G=(V,E)$ has a directed cycle~$C\subseteq E$, then
each corresponding move can be conformally decomposed into a lift of $b$ that corresponds to the directed multigraph
$(V,E\setminus C)$ and an element of $\ker_{\zz} \Bcal_{C_{3}, (p,2,r)}$ corresponding to the multigraph~$(V,C)$. 
By Lemma~\ref{lem:conformal-redundancy}, moves that possess such a conformal decomposition are redundant.
\end{proof}

\begin{example}
  Let $r=3$.  Up to symmetry the PF Gr\"obner basis contains a single move
  $b=\begin{smallpmatrix} 1 & -1 & 0 \\ -1 & 1 & 0 \end{smallpmatrix}$ for which $b' = (1,-1,0)$.  There are two acyclic
  directed multigraphs that satisfy the prescribed indegree and outdegree conditions, the graph with a single edge $1
  \to 2$ and the graph with two edges $1 \to 3$ and $3 \to 2$.  The corresponding lifts in this case are, in tableau
  notation,
  \begin{equation*}
    \scalebox{0.9}{$
    \begin{bmatrix}
      a 1 c_{1}  \\
      a 2 c_{2}
    \end{bmatrix}$} -
    \scalebox{0.9}{$
    \begin{bmatrix}
      a 1 c_{2}  \\
      a 2 c_{1}
    \end{bmatrix}$}
    \quad \mbox{  and  }  \quad
    \scalebox{0.9}{$
    \begin{bmatrix}
      a_{1} 1 c_1  \\
      a_{1} 2 c_3  \\
      a_{2} 1 c_3  \\
      a_{2} 2 c_2
    \end{bmatrix}$} -
    \scalebox{0.9}{$
    \begin{bmatrix}
      a_{1} 1 c_3  \\
      a_{1} 2 c_1  \\
      a_{2} 1 c_2  \\
      a_{2} 2 c_3
    \end{bmatrix}$}, \quad \text{for $a,a_{1},a_{2}\in[p]$, $c_{1}, c_{2}, c_{3} \in [3]$.}\tag*{\qed}
  \end{equation*}
\end{example}

Finally, we need to glue the lifts coming
from $\Gamma_{1}=[12][13]$ and $\Gamma_{2}=[24][34]$.  Lemma~\ref{lem:gluing-lifts} tells us to 
calculate $\Glues(m,m')$ for all pairs of lifts $m,m'$ of the same
element $g$ in the PF Gr\"obner basis.  However, it can happen that
some such glues are not actually needed in the resulting Gr\"obner basis and can be eliminated, as the following example shows.

\begin{example}\label{ex:badglue}
For $r=3$ consider the glued move
\begin{equation*}
  \scalebox{0.9}{$
  \begin{bmatrix}
    a_{1} 1 1 e_{1}  \\
    a_{1} 2 3 e_{1} \\ 
    a_{2} 1 3 e_{2} \\
    a_{2} 2 2 e_{2} 
  \end{bmatrix}$} -
  \scalebox{0.9}{$
  \begin{bmatrix}
    a_{1} 1 3 e_{1}  \\
    a_{1} 2 1 e_{1} \\
    a_{2} 1 2 e_{2} \\
    a_{2} 2 3 e_{2} 
  \end{bmatrix}$}
  \in
  \Glues\Bigg\{
  \scalebox{0.9}{$
  \begin{bmatrix}
    a_{1} 1 1  \\
    a_{1} 2 3  \\
    a_{2} 1 3  \\
    a_{2} 2 2  
  \end{bmatrix}$} -
  \scalebox{0.9}{$
  \begin{bmatrix}
    a_{1} 1 3  \\
    a_{1} 2 1  \\
    a_{2} 1 2  \\
    a_{2} 2 3  
  \end{bmatrix}$},
  \quad
  \scalebox{0.9}{$
  \begin{bmatrix}
    1 1 e_{1}  \\
    2 3 e_{1} \\ 
    1 3 e_{2} \\
    2 2 e_{2} 
  \end{bmatrix}$} -
  \scalebox{0.9}{$
  \begin{bmatrix}
    1 3 e_{1}  \\
    2 1 e_{1} \\
    1 2 e_{2} \\
    2 3 e_{2} 
  \end{bmatrix}$}
  \Bigg\}.
\end{equation*}
This move is the conformal decomposition of two degree $2$ moves (which are themselves glue moves)
namely
$$
\scalebox{0.9}{$
\begin{bmatrix}
a_{1} 1 1 e_{1}  \\
a_{1} 2 3 e_{1} \\ 
\end{bmatrix}$} -
\scalebox{0.9}{$
\begin{bmatrix}
a_{1} 1 3 e_{1}  \\
a_{1} 2 1 e_{1} \\
\end{bmatrix}$} \quad \mbox{ and } \quad
\scalebox{0.9}{$
\begin{bmatrix}
a_{2} 1 3 e_{2} \\
a_{2} 2 2 e_{2} 
\end{bmatrix}$} -
\scalebox{0.9}{$
\begin{bmatrix}
a_{2} 1 2 e_{2} \\
a_{2} 2 3 e_{2} 
\end{bmatrix}$}.
$$
By Lemma~\ref{lem:conformal-redundancy}, the move does not appear in a minimal Gr\"obner basis.  \qed
\end{example}

Applying the glue construction to all different pairs and throwing out the bad
combination in Example \ref{ex:badglue} produces the following general 
result.

\begin{theorem}
  \label{thm:MB-C4}
  Let $\Gamma = C_{4} := [12][13][24][34]$ and $d = (p,2,3,q)$, and let $\succeq_{\times}$ be as in
  Section~\ref{sec:TFP}.  A $\succeq_{\times}$-Gr\"obner basis of $\Bcal_{C_{4},d}$ consists of the moves from Example
  \ref{ex:codim0p23q} (from the associated codimension zero product) together with the necessary glue moves:
$$
\scalebox{0.9}{$
\begin{bmatrix}
a 1 c_{1} e \\
a 2 c_{2} e 
\end{bmatrix}$}-
\scalebox{0.9}{$
\begin{bmatrix}
a 1 c_{2} e \\
a 2 c_{1} e 
\end{bmatrix}$},
\scalebox{0.9}{$
\begin{bmatrix}
a_{1} 1  c_1 e_{1} \\
a_{1} 2  c_2 e_{2} \\
a_{2} 1  c_2 e_{3} \\
a_{2} 2  c_3 e_{1} \\
\end{bmatrix}$}-
\scalebox{0.9}{$
\begin{bmatrix}
a_{1} 1  c_2 e_{3} \\
a_{1} 2  c_1 e_{1} \\
a_{2} 1  c_3 e_{1} \\
a_{2} 2  c_2 e_{2} \\
\end{bmatrix}$},
\scalebox{0.9}{$
\begin{bmatrix}
a_{1} 1  c_1 e_{1} \\
a_{2} 2  c_2 e_{1} \\
a_{3} 1  c_2 e_{2} \\
a_{1} 2  c_3 e_{2} \\
\end{bmatrix}$}-
\scalebox{0.9}{$
\begin{bmatrix}
a_{3} 1  c_2 e_{1} \\
a_{1} 2  c_1 e_{1} \\
a_{1} 1  c_3 e_{2} \\
a_{2} 2  c_2 e_{2} \\
\end{bmatrix}$},
$$
where $a, a_{1}, a_{2}, a_{3} \in [p]$,  $c_{1}, c_{2}, c_{3} \in \{1,2,3\}$, and $e, e_{1}, e_{2}, e_{3} \in [q]$. 
\end{theorem}

This example provides us with an explicit instance of the finiteness
stabilization of the independent set theorem of \cite{HillarSullivant2012}.
In particular, because of the moves coming from the codimension zero product, we
see that the Markov basis stabilizes up to symmetry when $p = q = 6$.

The fact that the Gr\"obner basis in Theorem~\ref{thm:MB-C4} is square-free implies that the semigroup
$\Nb\Bcal_{C_{4}}$ is normal for $d=(p,2,3,q)$, see~\cite[Proposition~13.15]{Sturmfels1996:GBCP}.  More generally,
iterating the argument (see Section~\ref{sec:degree-bounds} below) shows that the semigroup $\Nb\Bcal_{K_{2,N}}$ of
the complete bipartite graph $K_{2,N}$ is normal for $d_{1}=2$, $d_{2}=3$.  As mentioned before, in general, the toric
fiber product does not preserve normality~\cite{KahleRauh2014:TFPs_and_SPs}.  


\subsection[Gluing K4 minus an edge]{Example: $K_{4}$ minus an edge}
\label{sec:K4-e}

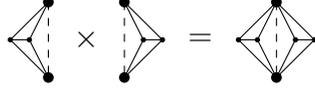
\begin{figure}
  \centering
  \begin{tikzpicture}[scale=0.5]
    \fill (1,1) circle (4pt);
    \fill (0,0) circle (2pt);
    \fill (0.5,0) circle (2pt);
    \fill (1,-1) circle (4pt);
    \draw (0,0) -- (1,1) -- (0.5,0) -- (1,-1) -- (0,0) -- (0.5,0);
    \draw[dashed] (1,1) -- (1,-1);
    \node at (2,0) {$\times$};
    \fill (3,1) circle (4pt);
    \fill (4,0) circle (2pt);
    \fill (3.5,0) circle (2pt);
    \fill (3,-1) circle (4pt);
    \draw (3.5,0) -- (3,1) -- (4,0) -- (3,-1) -- (3.5,0) -- (4,0);
    \draw[dashed] (3,1) -- (3,-1);
    \node at (5,0) {$=$};
    \fill (7,1) circle (4pt);
    \fill (6,0) circle (2pt);
    \fill (6.5,0) circle (2pt);
    \fill (7,-1) circle (4pt);
    \fill (8,0) circle (2pt);
    \fill (7.5,0) circle (2pt);
    \draw (6,0) -- (7,1) -- (6.5,0) -- (7,-1) -- (6,0) -- (6.5,0);
    \draw (7.5,0) -- (7,1) -- (8,0) -- (7,-1) -- (7.5,0) -- (8,0);
    \draw[dashed] (7,1) -- (7,-1);
  \end{tikzpicture}
  \caption{Gluing two copies of~$\tilde C_{4}$.  The dashed edges are the additional edges of the codimension zero product.
  }
  \label{fig:K4-e}
\end{figure}
In this section, we consider the problem of constructing a Markov basis for the complexes obtained from $\tilde C_{4} :=
[12][13][23][24][34]$ with $d = (2,2,2,2)$, by gluing multiple copies of $\tilde C_{4}$ together along the ``missing
edge''~$[14]$; see~Figure~\ref{fig:K4-e}.  This example serves two purposes: First, it demonstrates how our results can
be adjusted in the presence of holes in the projected fibers, if the structure of the holes is nice enough.  A more
complex example in which the projected fibers have holes (the complete bipartite graph $K_{3,N}$) is discussed at the
end of Section~\ref{sec:degree-bounds}.  Second, it illustrates that our procedure does not in general yield a minimal
Markov basis.  The main focus of this section lies on understanding the projected fibers.  Therefore, we do not write
out the final Markov basis explicitly, but we give a Markov basis of $\tilde C_{4}$ that satisfies the compatible
projection property in Proposition~\ref{prop:hatk4lift}.

Gluing binary hierarchical models along a missing edge is a codimension one toric
fiber product; see~\cite{EngstromKahleSullivant13:TFP-II} for further examples.
If the associated codimension zero semigroup were normal,
the Markov basis of $\ker_{\zz} \Bcal_{\tilde C_{4},d}$ would be slow-varying (Lemma~\ref{lem:slow-varying-lift})
and we could directly apply the results of \cite{EngstromKahleSullivant13:TFP-II}
to construct a Markov basis (Lemma~\ref{lem:slow-varying-TFP}).

The associated codimension-zero complex of $\tilde C_{4}$ 
is equal to the complete graph $K_{4}:=[12][13][14][23][24][34]$ (Proposition~\ref{prop:tfp-hierarchical}).
For $d=(2,2,2,2)$, a minimal Markov basis consisting of 60 moves can easily be computed using~\verb|4ti2|.
From it, a kernel Markov basis can be obtained using Theorem~\ref{thm:codimzero}.
We do not want to discuss the kernel Markov basis in more detail, but we focus on the PFI Markov basis and the lifting
process instead.

Next, we compute a PFI Markov basis.  The result is the following:
\begin{lemma}
  \label{lem:hatk4-PFI}
  A PFI Markov basis for gluing binary copies of $\tilde C_{4}$ along the missing edge is given by (in tableau notation
  and in tensor notation, respectively)
  \begin{equation*}
    \Gcal = \left\{
      \scalebox{0.9}{$
      \begin{bmatrix}
        00  \\ 11
      \end{bmatrix}$}
      -
      \scalebox{0.9}{$
      \begin{bmatrix}
        01  \\ 10
      \end{bmatrix}$},
      \scalebox{0.9}{$\begin{bmatrix}
        00 \\ 00 \\ 11 \\ 11
      \end{bmatrix}$}
      -
      \scalebox{0.9}{$\begin{bmatrix}
        01 \\ 01 \\ 10 \\ 10
      \end{bmatrix}$}
    \right\} \quad = \quad 
    \left\{
      \scalebox{0.9}{$\begin{pmatrix}
        +1 & -1 \\ -1 & +1
      \end{pmatrix}$}
      ,
      \scalebox{0.9}{$\begin{pmatrix}
        +2 & -2 \\ -2 & +2
      \end{pmatrix}$}
    \right\}.
  \end{equation*}
\end{lemma}

To prove this result, we need to study~$\nn \Bcal_{K_{4},d}$.  This semigroup is not normal, and so we cannot compute a
PF Markov basis directly as an inequality Markov basis.  We will prove Lemma~\ref{lem:hatk4-PFI} after describing the
single hole.

\begin{proposition}
  With $d = (2,2,2,2)$, the semigroup $\nn  \Bcal_{K_{4},d}$
  has a single hole ${\bf 1}$, which has a one in each component (that is, all pair margins are equal to one).
\end{proposition}

\begin{proof}
We follow the algorithm of~\cite{HemmeckeTakemureYoshida09:Computing_holes_in_semigroups}.
Direct computation with \verb|Normaliz| yields
that there is exactly one Hilbert basis element of the normalization of $\nn  \Bcal_{K_{4},d}$
which is not in $\nn  \Bcal_{K_{4},d}$, namely the vector~${\bf 1}$.
That means that ${\bf 1}$ is the unique
fundamental hole of~$\nn  \Bcal_{K_{4},d}$.  Any other hole of $\nn  \Bcal_{K_{4},d}$
must be of the form ${\bf 1} +f$ for some nonzero $f \in \nn  \Bcal_{K_{4},d}$.  Hence,
it suffices to check whether ${\bf 1} +f \in \nn \Bcal_{K_{4},d}$
for each generator $f$ of $\nn \Bcal_{K_{4},d}$.  By symmetry, we can check this for
any single generator, say $f = \Bcal_{K_{4},d}e_{0000}$ equals the first column of~$\Bcal_{K_{4},d}$.  In this case ${\bf 1} +f$
consists of all pair margins equal to $\begin{smallpmatrix} 2 & 1 \\ 1 & 1 \end{smallpmatrix}$.
The table
$$
v := e_{0001} + e_{0010} + e_{0100} + e_{1000} + e_{1111}
$$
has these pair margins (i.e.~$\mathbf{1}+f=\Bcal_{K_{4},d}v$), and so $\mathbf{1}+f$ is not a
hole. 
\end{proof}

\begin{proof}[Proof of Lemma~\ref{lem:hatk4-PFI}]
  $\Acal$ has codimension one, and $\phi(\Zker\Bcal_{\tilde C_{4}})=\Zb g$ is generated by the single
  move~$g:=\begin{smallpmatrix} +1 & -1 \\ -1 & +1\end{smallpmatrix}$.  Every projected fiber $\Fbf$ without hole is of
  the form~$\{u_{0} + kg : l\le k\le l'\}$.  Thus, the move $g$ suffices to connect~$\phi(\Fbf)$.

  Since $\mathbf{1}$ is the unique hole of $\Nb\Bcal_{K_{4},d}$, there is a single fiber $\Fbf(\Bcal_{\tilde C_{4},d},
  {\bf 1})$ that has a hole.  The fiber $\Fbf(\Bcal_{\tilde C_{4},d}, {\bf 1})$ consists of the following four tables:
  \begin{equation*}
    \scalebox{0.9}{$\begin{bmatrix}
      0000 \\
      1011 \\
      1101 \\
      0110
    \end{bmatrix}$}, \quad
    \scalebox{0.9}{$\begin{bmatrix}
      0001 \\
      1010 \\
      1100 \\
      0111
    \end{bmatrix}$}, \quad
    \scalebox{0.9}{$\begin{bmatrix}
      1000 \\
      0011 \\
      0101 \\
      1110
    \end{bmatrix}$}, \quad
    \scalebox{0.9}{$\begin{bmatrix}
      1001 \\
      0010 \\
      0100 \\
      1111
    \end{bmatrix}$}.
  \end{equation*}
  The projected fiber consists of the corresponding $[14]$-marginals:
  \begin{equation*}
    \phi(\Fbf(\Bcal_{\tilde C_{4},d}, {\bf 1}))
    = \left\{
      \scalebox{0.9}{$\begin{bmatrix}
        00 \\ 00 \\ 11 \\ 11
      \end{bmatrix}$}
      ,
      \scalebox{0.9}{$\begin{bmatrix}
        01 \\ 01 \\ 10 \\ 10
      \end{bmatrix}$}
    \right\}
    = \left\{
      \scalebox{0.9}{$\begin{pmatrix}
        20 \\ 02
      \end{pmatrix}$}
      ,
      \scalebox{0.9}{$\begin{pmatrix}
        02 \\ 20
      \end{pmatrix}$}
    \right\}.
  \end{equation*}
  Thus, to connect $\phi(\Fbf(\Bcal_{\tilde C_{4},d}, {\bf 1}))$, we need the move $2g=\begin{smallpmatrix} +2 & -2 \\ -2 &
    +2
  \end{smallpmatrix}$.

  The set $\{g,2g\}$ also connects all intersections of projected fibers: If the intersection only
  involves projected fibers without holes, then it is connected by~$g$.  Otherwise, the intersection is a subset of the
  two-element set~$\phi(\Fbf(\Bcal_{\tilde C_{4},d},\mathbf{1}))$ and thus connected by~$2g$.
\end{proof}

Next, we want to lift the PFI Markov basis.  We will not follow the lifting procedure of
Section~\ref{sec:lifting-GBs-algo} in detail, but we will just compute enough lifts to ensure the compatible projection
property.
First, we compute a Markov basis of~$\tilde C_{4}$, then we compare it with the lifts.
The graph $\tilde C_{4}$ was already studied in Section~\ref{sec:4-cycle}.
As in Example~\ref{ex:codim0p23q}, we obtain a Markov basis of $\ker_{\zz} \Bcal_{\tilde C_{4},d}$ from
Theorems~\ref{thm:codimzero} and~\ref{thm:mb2xpxq}:
\begin{equation*}
  \Mcal_{1} = \left\{
    \scalebox{0.9}{$\begin{bmatrix}
      0ab0 \\ 1ab1
    \end{bmatrix}$}
    -
    \scalebox{0.9}{$\begin{bmatrix}
      0ab1 \\ 1ab0
    \end{bmatrix}$}
    ,
    \scalebox{0.9}{$\begin{bmatrix}
      000a \\
      011b \\
      101c \\
      110d
    \end{bmatrix}$}
    -
    \scalebox{0.9}{$\begin{bmatrix}
      100a \\
      111b \\
      001c \\
      010d
    \end{bmatrix}$}
    ,
    \scalebox{0.9}{$\begin{bmatrix}
      a000 \\
      b110 \\
      c011 \\
      d101
    \end{bmatrix}$}
    -
    \scalebox{0.9}{$\begin{bmatrix}
      a001 \\
      b111 \\
      c010 \\
      d100
    \end{bmatrix}$}
  \right\}.
\end{equation*}
Under $\phi$ this Markov basis projects onto the set~$\{0\}\cup\pm\Gcal$.  In particular, $\Mcal_{1}$ is not
slow-varying.
One can show that $\Mcal_{1}$ lifts the first element 
$g=\left(\begin{smallmatrix}+1&-1\\-1&+1\end{smallmatrix}\right)$ of~$\Gcal$ using the
algorithm from Section~\ref{sec:lifting-GBs-algo}.  
However, $\Mcal_{1}$ does not lift the second element~$2g$.  
 In fact, the lift of $2g$ computed according to
Section~\ref{sec:lifting-GBs-algo} contains 75 binomials, among them the elements of $2\Mcal$ of degree up to
eight.  The following result shows that it suffices to work with lifts of degree at most four.  As it turns
out, these additional lifts are sums of two elements from $\Mcal_{1}$ of degree two.

\begin{proposition}\label{prop:hatk4lift}
The set of moves 
\begin{equation*}
  \Mcal = \Mcal_{1} \cup \left\{
  \scalebox{0.9}{$\begin{bmatrix}
    0ab0 \\
    1ab1 \\
    0cd0 \\
    1cd1
  \end{bmatrix}$}
  -
  \scalebox{0.9}{$\begin{bmatrix}
    0ab1 \\
    1ab0 \\
    0cd1 \\
    1cd0
  \end{bmatrix}$} :  a,b,c,d \in \{0,1\} \right\}
\end{equation*}
is a Markov basis of  $\Bcal_{\tilde C_{4},d}$ that satisfies the compatible projection property.
\end{proposition}

\begin{proof}
  Suppose that~$b\neq\mathbf{1},b'\neq\mathbf{1}$.  The projected fibers $\phi(\Fbf(\Bcal_{\tilde C_{4},d},b))$
  and~$\phi(\Fbf(\Bcal_{\tilde C_{4},d},b'))$ have no holes.
  As shown in the proof of Lemma~\ref{lem:hatk4-PFI}, the intersection $\phi(\Fbf(\Bcal_{\tilde
    C_{4},d},b))\cap\phi(\Fbf(\Bcal_{\tilde C_{4},d},b'))$ is connected by~$g$.  Since $\Mcal_{1}$ lifts~$g$,
  \begin{equation*}
    \phi(\Fbf(\Bcal_{\tilde C_{4},d},b)_{\Mcal_{1}})\cap\phi(\Fbf(\Bcal_{\tilde C_{4},d},b')_{\Mcal_{1}})
    = \phi(\Fbf(\Bcal_{\tilde C_{4},d},b))\cap\phi(\Fbf(\Bcal_{\tilde C_{4},d},b'))_{\{g\}}
  \end{equation*}
  is also connected.  
  Therefore, the compatible projection property is satisfied for this intersection.

  It remains to study intersections of the form $\phi(\Fbf(\Bcal_{\tilde C_{4},d},
  \mathbf{1}))\cap\phi(\Fbf(\Bcal_{\tilde C_{4},d}, b'))$ involving the hole.  If $b'=\mathbf{1}$, then
  $\phi(\Fbf(\Bcal_{\tilde C_{4},d},\mathbf{1})_{\Mcal_{1}})\cap\phi(\Fbf(\Bcal_{\tilde C_{4},d}, b')_{\Mcal_{1}}) =
  \phi(\Fbf(\Bcal_{\tilde C_{4},d},\mathbf{1})_{\Mcal_{1}})$ is connected, since $\Mcal_{1}$ is a Markov basis
  of~$\Bcal_{\tilde C_{4},d}$.  So assume that $b'\neq\mathbf{1}$.  Furthermore, we may assume that this intersection is
  non-empty.  Then $\Fbf(\Bcal_{\tilde C_{4},d},b')$ consists of elements with the same $[14]$-margins as~$\mathbf{1}$
  and with total entry sum equal to four.  Therefore, any lift of $2g$ that connects two elements of
  $\Fbf(\Bcal_{\tilde{C}_{4},d},b')$ has degree at most four.

  As $b'\neq\mathbf{1}$, the fiber $\Fbf(\Bcal_{\tilde C_{4},d},b')$ has no hole.  Suppose that there exist
  $v_{1},v_{2}\in\Fbf(\Bcal',b')$ such that $m=v_{1}-v_{2}$ is one of the degree four moves in~$\Mcal_{1}$.  Since
  $v_{1}$ and $v_{2}$ are also of degree four, $v_{1}=m^{+}$ and $v_{2}=m^{-}$.  One can check that in this case no
  other move of $\Mcal_{1}$ can be applied to~$v_{1}$ or $v_{2}$, and so $\Fbf(\Bcal_{\tilde
    C_{4},d},b')=\{v_{1},v_{2}\}$.  Hence, $\phi(\Fbf(\Bcal_{\tilde C_{4},d},\mathbf{1}))\cap\phi(\Fbf(\Bcal_{\tilde
    C_{4},d},b'))=\emptyset$.  Therefore, whenever $\phi(\Fbf(\Bcal_{\tilde
    C_{4},d},\mathbf{1}))\cap\phi(\Fbf(\Bcal_{\tilde C_{4},d},b'))$ is not empty, then $\Fbf(\Bcal_{\tilde C_{4},d},b')$
  is connected by the quadratic moves in~$\Mcal_{1}$.

  The intersection $\phi(\Fbf(\Bcal_{\tilde C_{4},d},\mathbf{1}))\cap\phi(\Fbf(\Bcal_{\tilde C_{4},d},b'))$ consists of at
  most two points.  If it 
  consists of one point, then it is connected.  Otherwise, if it consists of two points $u_{+1},u_{-1}$, then
  $u_{+1}-u_{-1}=\pm 2g$, and $\phi(\Fbf(\Bcal_{\tilde C_{4},d},b'))$ contains $u_{+1},u_{-1}$ and
  $u_{0}=\frac12(u_{+1}+u_{-1})$.  Any path from $u_{+1}$ to $u_{-1}$ in $\phi(\Fbf(\Bcal_{\tilde C_{4},d},b')_{\Mcal_{1}})$
  passes through~$u_{0}$.  To go from $u_{+1}$ to $u_{-1}$ directly, it suffices to add to $\Mcal_{1}$ all sums of two
  quadratic moves in~$\Mcal_{1}$.
\end{proof}


Using the results of Section~\ref{sec:TFP}, the Markov basis in Proposition~\ref{prop:hatk4lift} can be glued with
itself to compute a Markov basis of the solid graph 
on the right hand side of Figure~\ref{fig:K4-e} (together with the kernel Markov basis). %
In fact, as discussed in Section~\ref{sec:degree-bounds}, any number of copies of $\tilde C_{4}$ can be glued at the
missing edge~$[14]$.  The gluing procedure is straightforward, so we do not describe it in detail here.


\section{Finiteness results for iterated toric fiber products}
\label{sec:degree-bounds}

Forming the Markov basis of the toric fiber product can lead
to moves of larger degree than any of the moves in any
of the Markov bases that went into the construction.
However, we will show that, no matter how many factors
are involved in an iterated toric fiber product over the same
base $\cala$, if the degrees of the PFI Markov basis stabilize
and all other Markov bases have bounded degree, then
there exists a bound on the degree of the glued moves.
To prove this, we need to be precise about
what is meant by iterated toric fiber product and stabilization.

The toric fiber product $\Bcal\TFP\Bcal'$ is again $\Acal$-graded
in a natural way, by the map~$\xi$.  If $\Bcal''$ is another $\Acal$-graded integer matrix, then
\begin{equation*}
  (\Bcal\TFP\Bcal')\TFP\Bcal'' = \Bcal\TFP(\Bcal'\TFP\Bcal'').
\end{equation*}
In fact, our algorithm easily generalizes to the following related algorithm which is symmetric in the three matrices $\Bcal$, $\Bcal'$ and~$\Bcal''$: Let $\Gcal$ be a Markov basis of the family of sets
\begin{equation*}
 \{ \phi(\Fbf(\Bcal, b))\cap\phi'(\Fbf(\Bcal', b'))\cap\phi''(\Fbf(\Bcal'', b'')) :  
 b \in \nn\Bcal ,  b' \in \nn\Bcal' ,  b'' \in \nn\Bcal''  \}
\end{equation*}
Then a Markov basis of $\Bcal\TFP\Bcal'\TFP\Bcal''$ is given by the union of a Markov basis of the associated
codimension-zero product $\hat\Bcal\TFP[\hat\Acal]\hat\Bcal'\TFP[\hat\Acal]\hat\Bcal''$ and the set
\begin{equation*}
  \bigcup_{g\in\Gcal}\Glues(\Lifts_{\phi}(g),\Lifts_{\phi'}(g),\Lifts_{\phi''}(g)),
\end{equation*}
where
\begin{equation*}
  \Glues(f,g,h) := \Glues(f,\Glues(g,h)) = \Glues(\Glues(f,g),h).
\end{equation*}
Similarly, we can define the toric fiber powers $\bigTFP^{r}\Bcal$.


For any $v\in\Zb^{n}$ let $\deg(v) := \max\{ \|v^{+}\|_{1}, \|v^{-}\|_{1}\}$
(then $\deg(v)$ equals the degree of the binomial $x^{v^{+}}-x^{v^{-}}$
corresponding to~$v$; see Theorem~\ref{thm:MBthm}).
For $\Mcal\subseteq\Zb^{n}$ let $\deg(\Mcal) := \sup\{\deg(v) : v\in\Mcal\}$.  
For any family $\Fcal$ of subsets of $\Zb^{n}$ the \emph{Markov degree} $\mardeg(\Fcal)$ is the
minimum of $\deg(\Mcal)$ where $\Mcal$ ranges over all Markov
bases $\Mcal$ of $\Fcal$.  For  a matrix~$\Bcal$ we define $\mardeg(\Bcal):=\mardeg(\Fcal(\Bcal))$.
Our key lemma to obtain bounds on the Markov degrees of iterated toric fiber products is the following:
\begin{lemma}
  \label{lem:iterated-degree}
  Let $\Bcal_{1},\dots,\Bcal_{r}$ be integer matrices with $\Acal$-gradings $\phi_{1},\dots,\phi_{r}$, and consider
  lifts $m_{1}\in\Zker\Bcal_{1}$, \dots, $m_{r}\in\Zker\Bcal_{r}$ of the same move~$g\in\Zb^{\Acal}$.  Then the degree
  of any glued move $\tilde m\in\Glues(m_{1},\dots,m_{r})$ is bounded by
  \begin{equation*}
    \deg(\tilde m) \le \deg(g) + \deg\Big(\max_{i=1,\dots,r}(\phi_{i}(m_{i}^{+})-g^{+})\Big).
  \end{equation*}
  Here, $\max_{i=1,\dots,r}(\phi_{i}(m_{i}^{+})-g^{+})$ is a vector obtained
  by taking the maximum in each coordinate over all the vectors $\phi_{i}(m_{i}^{+})-g^{+}$.
\end{lemma}
\begin{proof}
  First, let $r=2$, and let $\tilde m$ be a glue of $m$ and~$m'$.
  In the notation of Section~\ref{sec:lifting-TFP},
  \begin{equation*}
    \xi(\tilde m^{+}) = \phi(\ol m^{+}) = \phi(m^{+}) + v^{+}.
  \end{equation*}
  Checking each component, one sees that
  \begin{equation*}
    \phi(m^{+}) + v^{+} = \max\big\{ \phi(m^{+}),\phi'(m^{\prime+}) \big\},
  \end{equation*}
  where $\max$ denotes the component-wise maximum.  Using induction, one sees that
  \begin{equation*}
    \xi^{r}(\tilde m^{+}) -  g^{+}  =  \max_{i = 1,\dots,r} (\phi_{i}(m_{i}^{+}) -  g^{+}),
  \end{equation*}
  where $\xi^{r}$ denotes the natural map $\Bcal_{1}\TFP\Bcal_{2}\TFP\dots\TFP\Bcal_{r}\to\Acal$.  Since $\xi^{r}(\tilde
  m)=g$ and since~$\xi^{r}$
  preserves the degree of positive vectors, $\deg(\tilde m^{+})-\deg(g^{+}) = \deg(\tilde
  m^{-})-\deg(g^{-})$ and
  \begin{multline*}
    \deg(\tilde m) = \max\{\deg(\tilde m^{+}),\deg(\tilde m^{-})\} = \deg(g) + \deg(\tilde m^{+}) - \deg(g^{+})
    \\
    = \deg(g) + \deg(\xi^{r}(\tilde m^{+})) - \deg(g^{+}) = \deg(g) + \deg(\xi^{r}(\tilde m^{+}) - g^{+}),
  \end{multline*}
  where the last equality uses that $\xi^{r}(\tilde m^{+})\ge g^{+}$ component-wise.
\end{proof}
As an example, we apply Lemma~\ref{lem:iterated-degree} to prove the following result:
\begin{theorem}
  \label{thm:degree-of-it-TFP}
  Let $\Gcal$ be a PFI Markov basis of the set of all intersected projected
  fibers 
  \begin{equation*}
    \Big\{
    \cap_{i = 1}^{r}\phi(\Fbf(\calb, b_{i}) : r \in \nn, b_{i} \in \nn \calb
    \Big\}.
  \end{equation*}
  If $\Gcal$ is finite, then
  there is a constant $C>0$ such that $\mardeg(\bigTFP^{r}\Bcal)\le C$ for any~$r>0$.
\end{theorem}
\begin{proof}
  Let $\hat\Mcal$ be a Markov basis of 
  $\Bcal^{\phi}$.  The degree of the Markov
  basis of the associated co\-di\-men\-sion-zero toric fiber product is bounded by $\max \{ 2, \deg(\hat\Mcal)\}$.  To
  prove the statement, it remains to find a bound for the glued moves that is independent of~$r$.  Such a bound is given
  by Lemma~\ref{lem:iterated-degree}.
\end{proof}
The proof of Theorem~\ref{thm:degree-of-it-TFP} is constructive in the sense that a Markov basis of the toric fiber
powers can be obtained explicitly by following the constructions discussed in this paper.  In the same way, a numerical
value for the constant $C$ can be computed explicitly.  The same remark holds for the other results of this section.

\begin{corollary}
  \label{cor:degree-of-it-TFP}
  Let $\Bcal$ be an $\Acal$-graded integer matrix such that $\nn\Bcal^{\phi}$ is normal.
  Then $\sup_{r\in\Nb}\mardeg(\bigTFP^{r}\Bcal)$ is finite.
\end{corollary}

\begin{proof}
  Let $D$ be an integer matrix such that for all $b\in\Nb\Bcal$ 
  there exists $c=c(b)$ 
  such that
  \begin{equation*}
    \phi(\Fbf(\Bcal,b)) 
    = \big\{
      v\in\Zb^{\Acal} : D v \ge c
    \big\}.
  \end{equation*}
  This implies that for all $b_{1},\dots,b_{r}\in\Nb\Bcal$, 
  \begin{equation*}
    \phi(\Fbf(\Bcal,b_{1}))\cap\dots\cap\phi(\Fbf(\Bcal,b_{r})) = \big\{
    v\in\Zb^{\Acal} : D v \ge \max_{i=1}^{r}c(b_{i})
    \big\}.
  \end{equation*}
  This shows that an inequality Markov basis of $D$ is a finite PFI Markov basis that works for any toric fiber power
  $\bigTFP^{r}\Bcal$.  Therefore, we can apply Theorem~\ref{thm:degree-of-it-TFP}.
\end{proof}

Lemma~\ref{lem:iterated-degree} can be applied to
more general situations.  The crucial point is that there needs to be a single finite Markov basis.
For example, Corollary~\ref{cor:degree-of-it-TFP} holds true if there are only finitely many holes.  As further examples, we
mention the following result:

\begin{theorem}
  Let $\Bcal_{1},\dots,\Bcal_{s}$ be $\Acal$-graded matrices such that the
  semigroups $\nn \Bcal_{1}^{\phi_{1}}$,$\ldots$, $\nn \Bcal_{s}^{\phi_{s}}$ are normal.  Then there is a constant $C\in\Nb$ such that
  \begin{equation*}
    \mardeg \left( (\bigTFP^{r_{1}}\Bcal_{1})\TFP(\bigTFP^{r_{2}}\Bcal_{2})\TFP\dots\TFP(\bigTFP^{r_{s}}\Bcal_{s}) \right) \le C
  \end{equation*}
  for all $r_{1},\dots,r_{s}\in\Nb$.
\end{theorem}

The same ideas can be applied in the specific situation of hierarchical models, taking advantage of the situations where
we know that the semigroup of the associated codimension-zero product is normal.  For example:

\begin{corollary}
\label{cor:K_2N-bound}
  Consider the complete bipartite graph~$K_{2,N}$ with $2+N$ vertices.
  For each $k \in \nn$, there is a constant  $C(k) \in\Nb$ such that
  for all $N \in \nn $ and $d\in\Nb^{N}$ with $d_{1} = 2$ and  $d_{2} = k$, 
  $\mardeg(\Bcal_{K_{2,N},d})\le C(k)$.
\end{corollary}

\begin{proof}
$K_{2,N}$ is obtained by gluing $N$ paths of $3$ nodes on the pair of 
end-points of the missing edge.  With our conditions on $d_{1}$, each such path
corresponds to a hierarchical model of $K_{2,1}$ with $d = (2,k, d_{i})$.  The
associated codimension zero product is a product of cycles $K_{3}$ (Proposition~\ref{prop:tfp-hierarchical}), and the
semigroup of $K_{3}$ is normal for our choice of parameters by Theorem~\ref{thm:C3-normal}.  More precisely, by
Proposition~\ref{prop:3cycleineq}, the projected fibers have an inequality description that is independent of~$d_{i}$.
Therefore, in this situation, there exists a finite inequality Markov basis~$\Gcal$ (any solution of
Problem~\ref{prob:sumsineqMB} with $t=d_{2}-1$) that can be used as a PFI Markov basis, independent of~$r$ and the
choice of $d_{3},\dots,d_{2+N}$.

Proposition~\ref{prop:lifting-3string} gives a combinatorial description of the lifts.  In
particular, there is a finite number of combinatorial types of lifts.  This finite number is independent of~$d_{i}$.
Moreover, if $m$ lifts~$g$, then the quantity $\phi(m^{+})-g^{+}$ only depends on the combinatorial type of the lift.
Therefore, there is a constant $d^{*}(g)$ with
\begin{equation*}
  \deg\Big(\max_{m\in\Lifts(g)}(\phi(m^{+})-g^{+})\Big) \le d^{*}(g),
\end{equation*}
and this bound is again independent of~$d_{i}$.  By Lemma~\ref{lem:iterated-degree}, the degree of any glued move is upper
bounded by~$\max_{g\in\Gcal}\deg(g)+ d^{*}(g)$.
Therefore, the statement follows as in the proof of Theorem~\ref{thm:degree-of-it-TFP}.
\end{proof}

Note that results from~\cite{HillarSullivant2012} imply a finiteness result
of this type for any fixed $N$.  The novelty of Corollary~\ref{cor:K_2N-bound}
is that a bound holds regardless of~$N$.

It is a nontrivial problem to determine the number $C(k)$ from Corollary \ref{cor:K_2N-bound}.
For $k = 2$, a Markov basis was explicitly calculated in~\cite{KRS14:Positive_Margins_and_Prim_Dec},
and the result there implies that $C(2) = 4$.
Careful reasoning about the lifting procedure for the PF Markov basis that is described in
Proposition~\ref{prop:lifting-3string} can be used to produce bounds on $C(k)$ in other
instances.  For example, it is not difficult to show that $C(3) = 6$.  We do not
know the growth rate of~$C(k)$.

The conditions on $d$ in the statement of Corollary~\ref{cor:K_2N-bound} are chosen such that all factors arising in the
toric fiber product have normal semigroups.  We conjecture that this assumption is not necessary; i.e.~we conjecture
that there is a function $C(d_{1},d_{2})\in\Nb$ such that $\deg(\Bcal_{K_{2,N},d})\le C(d_{1},d_{2})$.  More
generally, we formulate the following conjecture:
\begin{conjecture}
  Let $\Bcal_{1},\dots,\Bcal_{s}$ be arbitrary $\Acal$-graded matrices.  Then there is a constant $C\in\Nb$ such that
  \begin{equation*}
    \mardeg\left( (\bigTFP^{r_{1}}\Bcal_{1})\TFP(\bigTFP^{r_{2}}\Bcal_{2})
    \TFP\dots\TFP(\bigTFP^{r_{s}}\Bcal_{s}) \right) \le C
  \end{equation*}
  for all $r_{1},\dots,r_{s}\in\Nb$.  
\end{conjecture}



\begin{figure}
  \centering
  \begin{tikzpicture}[scale=0.5]
    \node at (-1,1) {a)};
    \fill (0,1) circle (4pt);
    \fill (0,0) circle (4pt);
    \fill (0,-1) circle (4pt);
    \fill (1,1) circle (2pt);
    \draw (1,1) edge (0,1) edge (0,0) edge (0,-1);
    \node at (2,0) {$\times$};
    \fill (3,1) circle (4pt);
    \fill (3,0) circle (4pt);
    \fill (3,-1) circle (4pt);
    \fill (4,0) circle (2pt);
    \draw (4,0) edge (3,1) edge (3,0) edge (3,-1);
    \node at (5,0) {$\times$};
    \fill (6,1) circle (4pt);
    \fill (6,0) circle (4pt);
    \fill (6,-1) circle (4pt);
    \fill (7,-1) circle (2pt);
    \draw (7,-1) edge (6,1) edge (6,0) edge (6,-1);
    \node at (8,0) {$=$};
    \fill (9,1) circle (4pt);
    \fill (9,0) circle (4pt);
    \fill (9,-1) circle (4pt);
    \fill (10,1) circle (2pt);
    \fill (10,0) circle (2pt);
    \fill (10,-1) circle (2pt);
    \draw (10,1) edge (9,1) edge (9,0) edge (9,-1);
    \draw (10,0) edge (9,1) edge (9,0) edge (9,-1);
    \draw (10,-1) edge (9,1) edge (9,0) edge (9,-1);
  \end{tikzpicture}
  \hfil
  \begin{tikzpicture}[scale=0.5]
    \node at (-1,1) {b)};
    \filldraw[fill=gray] (0,1) -- (0.5,0) -- (0,-1) -- cycle;
    \fill (0,1) circle (4pt);
    \fill (0.5,0) circle (4pt);
    \fill (0,-1) circle (4pt);
    \fill (1.3,0) circle (2pt);
    \draw (1.3,0) edge (0,1) edge (0.5,0) edge (0,-1);
  \end{tikzpicture}
  \caption{a) Gluing three three-stars to obtain~$K_{3,3}$.
    b) The associated codimension-zero factor~$\tilde K_{4}$. }
  \label{fig:K33}
\end{figure}
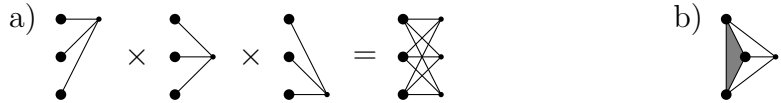

As an example we bound the Markov basis of the complete bipartite graph $K_{3,N}$ with binary nodes.  $K_{3,N}$ can
be obtained by gluing $N$ three-stars; see Figure~\ref{fig:K33}\,a).  For brevity, we just summarize the main results
here and refer to~\cite{RauhSullivant14:K3N} for the details.

The Markov basis of the associated codimension-zero product arises by lifting moves from the Markov basis of the
codimension-zero factor~$\tilde{K}_{4}$; see Figure~\ref{fig:K33}\,b).  The Markov basis of $\tilde{K}_{4}$ has 20
elements of degrees four and six.  As shown in \cite[Section~5]{RauhSullivant14:K3N}, the moves that arise by gluing
these elements are redundant in view of the moves that arise by lifting the PFI Markov basis.

To understand the projected fibers, we need to understand the semigroup of~$\tilde K_{4}$.  Following the algorithm
from~\cite{HemmeckeTakemureYoshida09:Computing_holes_in_semigroups} we find the following: This semigroup is not normal.
Even worse, it has infinitely many holes.  Fortunately, within the projected fibers the holes are vertices.  Therefore,
the holes can be separated from their projected fibers by linear inequalities.  To be precise, there are two linear
forms $l_{1},l_{2}$ with the following property: If $h$ is a hole of a projected fiber~$\phi(\Fbf)$, then
$l_{i}(h) < \min\{ l_{i}(u) : u\in\phi(\Fbf) \}$ for some~$i$.
This allows to give an inequality description of the projected fibers.  The corresponding inequality Markov basis can be
used as a PFI Markov basis.

The inequality Markov basis consists of 16 moves of degrees in four symmetry classes, two of degree two and two of
degree four.  As shown in \cite[Section~4]{RauhSullivant14:K3N}, lifting increases the degree by two.  Using
Lemma~\ref{lem:iterated-degree}, one can see that gluing different lifts of the same move leads to moves of degree at
most~12.  In fact, for any lift $m$ of~$g$, the tableau $\phi(m^{+})-g^{+}$ is \emph{square-free}, that is, it does not
contain two identical rows.  Therefore, for any glued lift $\tilde m$ of~$g$, the tableau $\phi(\tilde m^{+})-g^{+}$
will also be square free, and hence of degree at most~8.  Therefore, $\deg(\tilde m)\le\deg(g)+8\le 12$.
In total, we obtain the following result:
\begin{theorem}
  For any~$N$, the Markov degree of the binary hierarchical model of the complete bipartite graph $K_{3,N}$ is at most~$12$.
\end{theorem}


\section*{Acknowledgments}

Johannes Rauh was supported by the VW Foundation.
Seth Sullivant was partially supported by the David and Lucille Packard Foundation and the US National Science Foundation (DMS 0954865).


\bibliographystyle{elsarticle-num}
\bibliography{HighCodim}

\end{document}